\documentclass [11pt]{amsart}
\usepackage {amsmath, amssymb, amscd, mathrsfs, graphicx, color, url,epsf}
\usepackage[text={6.5in,9in},centering,letterpaper,dvips]{geometry}

\newtheorem {theorem}{Theorem}[section]
\newtheorem {lemma}[theorem]{Lemma}
\newtheorem {proposition}[theorem]{Proposition}
\newtheorem {corollary}[theorem]{Corollary}
\newtheorem {conjecture}[theorem]{Conjecture}
\newtheorem {definition}[theorem]{Definition}
\newtheorem {question}[theorem]{Question}
\newtheorem {remark}[theorem]{Remark}
\newtheorem {example}[theorem]{Example}

\newtheorem{assumption}[theorem]{Assumption}

\newcommand\eps{\varepsilon}
\renewcommand\div{R}
\newcommand{\dual}{\vee}

\def\zz {{\mathbb{Z}}}
\def\rr {{\mathbb{R}}}
\def\cc {{\mathbb{C}}}
\def\qq {{\mathbb{Q}}}
\def\NN {{\mathbb{N}}}
\def\aa {{\mathscr{A}}}
 
\def\R {{\rr}}
\def\Z {{\zz}}
\def\C {{\cc}}

\def\g {{\mathfrak{g}}}
\def\c {c}

\def\uu {{\mathcal{U}}}

\def\diag {{\operatorname{diag}}}
\def\gauge {{\mathscr{G}}}
\def\tt {{\mathfrak{t}}}
\def\t {{\mathfrak{t}}}

\def\pp {{\mathbb{P}}}

\def\alcove {{\mathfrak{A}}}

\def\del {{\partial}}

\def\Ad {{\operatorname{Ad}}}
\def\red {{\operatorname{red}}}
\def\reg {{\operatorname{reg}}}

\def\dim {{\operatorname{dim}}}
\def\ker {{\operatorname{Ker}}}
\def\im {{\operatorname{Im}}}
\def\tr {{\operatorname{Tr} \ }}
\def\diff {{\operatorname{d}}}
\def\fin\qedhere
\def\pr {{\text{pr}}}

\def\hsi{{HSI}}

\def\jj{{\mathcal{J}}}
\def\olde {{\omega_{\epsilon}}}
\def\oldo {{\omega_0}}
\def\oe { \omega}
\def\ol { \overline}
\def\ul { \underline}
\def\oo { \tilde{\omega}}
\def\zed {{R}}
\def\tad {{T^{\operatorname{ad}}}}

\def\ww {{\mathcal{W}}}

\def\mm {{\mathscr{M}}}
\def\nn {{\mathscr{N}}}

\def\opens {{\mathcal{W}}}

\def\lieg {{\mathfrak{g}}}
\def\liek {{\mathfrak{k}}}

\def\gad {{G^{\operatorname{ad}}}}
\def\ext {{\lieg}}
\def\diag {\on{diag}}

\def\omu {{\mathcal{O}_{\mu}}}
\def\ol {{\mathcal{O}_{\lambda}}}

\def\tw {{\operatorname{tw}}}
\def\on{\operatorname}

\def\M {\mathfrak{M}}
\def\MM {\mathfrak{M}}
\def\J {\mathcal{J}}
\def\ti {\tilde}
\def\jj {\mathcal{J}}

\newcommand\bra[1]{ < \kern-.7ex {#1} \kern-.7ex >} 
\def\jt {\jj_t}
\def\P {\mathbb{P}}
\def\g {{\lieg}}
\def\d {{\operatorname{d}}}

\begin{document}

\title{Floer homology on the extended moduli space}

\author[Ciprian Manolescu]{Ciprian Manolescu}
\thanks {CM was partially supported by the NSF grant DMS-0852439 and a Clay Research Fellowship.}
\address {Department of Mathematics, UCLA, 520 Portola Plaza\\ 
Los Angeles, CA 90024}
\email {cm@math.ucla.edu}

\author[Christopher Woodward]{Christopher Woodward}
\address {Mathematics-Hill Center, Rutgers University, 110 Frelinghuysen 
Road\\ Piscataway, NJ 08854}
\email {ctw@math.rutgers.edu}

\thanks{CW was partially supported by the NSF grant DMS-060509 and
  DMS-0904358}

\begin {abstract}
Starting from a Heegaard splitting of a three-manifold, we use
Lagrangian Floer homology to construct a three-manifold invariant, in
the form of a relatively $\zz/8\zz$-graded abelian group. Our
motivation is to have a well-defined symplectic side of the
Atiyah-Floer Conjecture, for arbitrary three-manifolds. The symplectic
manifold used in the construction is the extended moduli space of flat
$SU(2)$-connections on the Heegaard surface. An open subset of this
moduli space carries a symplectic form, and each of the two
handlebodies in the decomposition gives rise to a Lagrangian inside
the open set. In order to define their Floer homology, we compactify
the open subset by symplectic cutting; the  resulting manifold is only semipositive, but we show that
one can still develop a version of Floer homology in this setting.
\end {abstract}

\maketitle

\section {Introduction}

Floer's instanton homology \cite{Floer} is an invariant of integral
homology three-spheres $Y$ which serves as target for the relative
Donaldson invariants of four-manifolds with boundary; see
\cite{DonaldsonBook}. It is defined from a complex whose generators
are (suitably perturbed) irreducible flat connections in a trivial
$SU(2)$-bundle over $Y,$ and whose differentials arise from counting
anti-self-dual $SU(2)$-connections on $Y \times \rr.$ There is also a
version of instanton Floer homology using connections in
$U(2)$-bundles with $c_1$ odd (\cite{FloerTriangles},
\cite{BraamDonaldson}), an equivariant version (\cite{AusBra},
\cite{AusBra2}), and several other variants which use both irreducible
and reducible flat connections \cite{DonaldsonBook}. More recently,
Kronheimer and Mrowka \cite{KMSutures} have developed instanton
homology for sutured manifolds; a particular case of their theory
leads to a version of instanton homology that can be defined for
arbitrary closed three-manifolds.

In another remarkable paper \cite{FloerLagrangian}, Floer associated a
homology theory to two Lagrangian submanifolds of a symplectic
manifold, under suitable assumptions.  This homology is defined from a
complex whose generators are intersection points between the two
Lagrangians, and whose differentials count pseudo-holomorphic
strips. The Atiyah-Floer Conjecture \cite{AtiyahFloer} states that
Floer's two constructions are related: for any decomposition of the
homology sphere $Y$ into two handlebodies glued along a Riemann
surface $\Sigma,$ instanton Floer homology should be the same as the
Lagrangian Floer homology of the $SU(2)$-character varieties of the
two handlebodies, viewed as subspaces of the character variety of
$\Sigma.$

As stated, an obvious problem with the Atiyah-Floer Conjecture is that
the symplectic side is ill-defined: due to the presence of reducible
connections, the $SU(2)$-character variety of $\Sigma$ is not
smooth. One way of dealing with the singularities is to use a version
of Lagrangian Floer homology defined via the symplectic vortex
equations on the infinite-dimensional space of all connections. This
approach was pursued by Salamon and Wehrheim, who obtained partial
results towards the conjecture in this set-up; see \cite{AFSalamon},
\cite{WehrheimJSG}, \cite{SalamonWehrheim}. Another approach is to
avoid reducibles altogether by using nontrivial $PU(2)$-bundles
instead. This road was taken by Dostoglou and Salamon \cite{DostSal},
who proved a variant of the conjecture for mapping tori.

The goal of this paper is to construct another candidate that could
sit on the symplectic side of the (suitably modified) Atiyah-Floer
Conjecture.  

Here is a short sketch of the construction. Let $\Sigma$ be a Riemann
surface of genus $h \geq 1,$ and $z \in \Sigma$ a base point. The
moduli space $\mm(\Sigma)$ of flat connections in a trivial
$SU(2)$-bundle over $\Sigma$ can be identified with the character
variety $\{\rho:\pi_1(\Sigma) \to SU(2)\}/PU(2).$ The moduli space
$\mm(\Sigma)$ is typically singular. However, Jeffrey \cite{Jeffrey}
and, independently, Huebschmann \cite{Huebschmann}, showed that
$\mm(\Sigma)$ is the symplectic quotient of a different space, called
the extended moduli space, by a Hamiltonian $PU(2)$-action. The
extended moduli space is naturally associated not to $\Sigma,$ but to
$\Sigma',$ a surface with boundary obtained from $\Sigma$ by deleting
a small disk around $z$. The extended moduli space has an open smooth
stratum, which Jeffrey and Huebschmann equip with a natural closed
two-form. This form is nondegenerate on a certain open set
$\nn(\Sigma'),$ which we take as our ambient symplectic manifold. In
fact, $\nn(\Sigma')$ can also be viewed as an open subset of the
Cartesian product $SU(2)^{2h} \cong \{\rho: \pi_1(\Sigma') \to SU(2)
\}.$ More precisely, if we pick $2h$ generators for the free groups
$\pi_1(\Sigma'),$ we can describe this subset as
$$  \nn (\Sigma') =  \{(A_1, B_1, \dots, A_h, B_h) \in SU(2)^{2h}\ | \ \prod_{i=1}^h [A_i, B_i] \neq -I \}.  $$

Consider a Heegaard decomposition of a three-manifold $Y$ as $Y=H_0
\cup H_1$, where the handlebodies $H_0$ and $H_1$ are glued along
their common boundary $\Sigma.$ There are smooth Lagrangians $L_i = \{
\pi_1(H_i) \to SU(2) \} \subset \nn(\Sigma'),$ for $i=0,1$. In order
to take the Lagrangian Floer homology of $L_0$ and $L_1$, care must be
taken with holomorphic strips going out to infinity; indeed, the
symplectic manifold $\nn(\Sigma')$ is not weakly convex at
infinity. Our remedy is to compactify $\nn(\Sigma')$ by (non-abelian)
symplectic cutting. The resulting manifold $\nn^c(\Sigma')$ is the
union of $\nn(\Sigma')$ and a codimension two submanifold $\div.$ A
new problem shows up here, because the natural two-form $\oo$ on
$\nn^c(\Sigma')$ has degeneracies on $\div$. Neverthless,
$(\nn^c(\Sigma'), \oo)$ is monotone, in a suitable sense. One can
deform $\oo$ into a symplectic form $\oe,$ at the expense of losing
monotonicity. We are thus led to develop a version of Lagrangian Floer
theory on $\nn^c(\Sigma')$ by making use of the interplay between the
forms $\oo$ and $\oe.$ Our Floer complex uses only holomorphic disks
lying in the open part $\nn(\Sigma')$ of $\nn^c(\Sigma')$. We show
that, while holomorphic strips with boundary on $L_0$ and $L_1$ can go
to infinity in $\nn(\Sigma'),$ they do so only in high codimension,
without affecting the Floer differential. The resulting Floer homology
group is denoted
$$\hsi(\Sigma; H_0, H_1) = HF(L_0, L_1 \text{ in } \nn(\Sigma')),$$
and admits a relative $\zz/8\zz$-grading. We call it {\em symplectic instanton homology}.

Using the theory of Lagrangian correspondences and
pseudo-holomorphic quilts developed in Wehrheim-Woodward and
\cite{WehrheimWoodward} and Lekili-Lipyanskiy
\cite{ll:geom}, we prove: 

\begin {theorem}
\label {thm:Invariance}
The relatively $\zz/8\zz$ graded group $\hsi(Y) = \hsi(\Sigma; H_0, H_1)$ is an invariant of the three-manifold $Y$. 
\end {theorem}

Strictly speaking, if we are interested in canonical isomorphisms, then the symplectic instanton homology also depends on the base point $z \in \Sigma \subset Y$: as $z$ varies inside $Y$, the
corresponding groups form a local system. However, we drop $z$ from
notation for simplicity.

Let us explain how we expect $\hsi(Y)$ to be related to the
traditional instanton theory on 3-manifolds. We restrict our attention
to the original set-up for Floer's instanton theory $I(Y)$ from
\cite{Floer}, when $Y$ is an integral homology sphere. It is then
decidedly not the case that $\hsi(Y)$ coincides with Floer's theory;
for example, we have $\hsi(S^3) \cong \zz$, but $I(S^3) = 0.$
Nevertheless, in \cite[Section 7.3.3]{DonaldsonBook}, Donaldson
introduced a different version of instanton homology, a
$\zz/8\zz$-graded vector field over $\qq$ denoted $\widetilde {HF}$,
which satisfies $\widetilde {HF}(S^3) \cong \qq$. (Floer's theory $I$
is denoted $HF$ in \cite{DonaldsonBook}.) We state the following
variant of the Atiyah-Floer Conjecture:

\begin {conjecture}
For every integral homology sphere $Y$, the symplectic instanton
homology $\hsi(Y) \otimes \qq$ and the Donaldson-Floer homology
$\widetilde {HF}(Y)$ from \cite{DonaldsonBook} are isomorphic, as
relatively $\zz/8\zz$-graded vector spaces.
\end {conjecture}

Alternatively, one could hope to relate $\hsi$ to the sutured version
of instanton Floer homology developed by Kronheimer and Mrowka in
\cite{KMSutures}. More open questions, and speculations along these
lines, are presented in Section~\ref{sec:comparisons}.

\medskip \noindent {\bf Acknowledgments.} We would like to thank Yasha
Eliashberg, Peter Kronheimer, Peter Ozsv\'ath, Tim Perutz, and Michael
Thaddeus for some very helpful discussions during the preparation of
this paper.  Especially, we would like to thank Ryszard Rubinsztein
for pointing out an important mistake in an earlier version of this
paper (in which topological invariance was stated as a conjecture).

\section {Floer homology}
\label {sec:floer}

\subsection {The monotone, nondegenerate case}
\label {sec:mn}

Lagrangian Floer homology was originally constructed in
\cite{FloerLagrangian} under some restrictive conditions, and later
generalized by various authors to many different settings. We review
here its definition in the monotone case, due to Oh \cite{OhMonotone,
  OhTransversality}, together with a discussion of orientations
following Fukaya-Oh-Ohta-Ono \cite{FOOO}.

Let $(M, \omega)$ be a compact connected 
symplectic manifold.  We denote by $\J(M,\omega)$ the space of
compatible almost complex structures on $(M, \omega)$, and by
$\jt(M,\omega)=C^{\infty}([0,1], \J(M, \omega))$ the space of {\em
  time-dependent} compatible almost complex structures.  Any
compatible almost complex structure $J$ defines a complex structure on
the tangent bundle $TM$. Since $\J(M, \omega)$ is contractible, the
first Chern class $c_1(TM) \in H^2(M,\Z)$ depends only on $\omega,$
not on $J$. The minimal Chern number $N_M$ of $M$ is defined as the
positive generator of the image of $c_1(TM) : \pi_2(M) \to \zz.$

\begin {definition}
\label {def:monotone}
Let $(M, \omega)$ be a symplectic manifold. $M$ is called {\em
  monotone} if there exists $\kappa > 0$ such that
$$ [\omega] = \kappa \cdot c_1(TM).$$
In that case, $\kappa$ is called the monotonicity constant.
\end {definition}

\begin {definition}
A Lagrangian submanifold $L \subset (M, \omega)$ is called {\em
  monotone} if there exists a constant $\kappa > 0$ such that
$$2 [\omega]|_{\pi_2(M, L)} = \kappa \cdot \mu_L,$$ where $\mu_L: \pi_2(M,
L) \to \zz$ is the Maslov index.
\end {definition}

Necessarily if $L$ is monotone then $M$ is monotone with the same
monotonicity constant.  
The minimal Maslov number $N_L$ of a monotone Lagrangian $L$ is
defined as the positive generator of the image of $\mu_L$ in $\zz.$

From now on we will assume that $M$ is monotone with monotonicity constant $\kappa$ 
and that we are given
two closed, simply connected Lagrangians $L_0,L_1 \subset M.$ These
conditions imply that $L_0$ and $L_1$ are monotone with the same monotonicity constant and 
$$ N_{L_0} = N_{L_1} = 2N_M.$$

We assume that $N_M > 1,$ and denote $N=2N_M \geq 4.$ We also assume that $w_2(L_0) = w_2(L_1) = 0.$

After a small Hamiltonian
perturbation we can arrange so that the intersection $L_0 \cap L_1$ is
transverse.  Let $(J_t)_{0 \leq t \leq 1} \in \jt(M,\omega)$. For any
$x_\pm \in L_0 \cap L_1$ we denote by $\tilde{\MM}(x_+,x_-)$ the space
of {\em Floer trajectories} (or {\em $J_t$-holomorphic strips}) from $x_+$ to $x_-$, i.e., finite energy solutions
to Floer's equation
\begin{equation} 
\label{floer}
\left\{
\begin{aligned}
 & u: \rr \times [0,1] \rightarrow M, \\
 & u(s,j) \in L_j, \quad j = 0,1, \\
 & \partial_s u + J_t(u) \partial_t u = 0, \\
 & \lim_{s \rightarrow  \pm \infty} u(s,\cdot) = x_\pm
\end{aligned}
\right.
\end{equation}
Let $\MM(x_+,x_-)$ denote the quotient of $\tilde{\MM}(x_+,x_-)$ by
the translational action of $\R$. For $(J_t)_{0 \leq t \leq 1}$ chosen
from a comeagre\footnote{A subset of a topological space is {\it comeagre} if it is
  the intersection of countably many open dense subsets. Many authors
  use the term ``Baire second category'', which however denotes more
  generally subsets that are not meagre, i.e.\ not the complement of a
  comeagre subset. See for example \cite[Chapter 7.8]{royden}.
  }  
 subset
$\jt^\reg(L_0,L_1) \subset \jt(M,\omega)$ of {\em
  $(L_0,L_1)$-regular}, time-dependent compatible almost complex
structures, $\MM(x_+,x_-)$ is a smooth, finite dimensional manifold
with dimension at non-constant $ u\in \MM(x_+,x_-)$ given by $\dim
\ T_u \MM(x_+,x_-) = I(u) - 1$.  We denote by $\MM(x_+,x_-)_d$ the
subset with $I(u) - 1 = d$ (note that $\MM(x_+,x_-)_{-1}$ is non-empty
if $x_+ = x_-$.)  As explained in Oh \cite{OhMonotone} after shrinking
$\jt^\reg(L_0,L_1)$ further we may assume that $\MM(x_+,x_-)_0$ is
finite and $\MM(x_+,x_-)_1$ is compact up to breaking of trajectories:
\begin{equation} \label{break} 
 \partial \MM(x_+,x_-)_1 = \bigcup_{y \in L_0 \cap L_1} \MM(x_+,y)_0
 \times \MM(y,x_-)_0 .\end{equation}
The condition that the Lagrangians have vanishing $w_2$ is used in
defining orientations on the moduli spaces, compatible with the
identity \eqref{break}.  The Floer chain complex is then defined to be
the free abelian group generated by the intersection points,
$$ CF(L_0, L_1) = \bigoplus_{x \in L_0 \cap L_1} \Z \bra{x} .$$
The Floer differential is
$$ \del \bra{x_+} = \sum_{u \in \MM(x_+,x_-)_0} \eps(u) \bra{x_-} $$
where $\eps(u) \in \{ \pm 1 \}$ is the sign comparing the orientation
of the moduli space to the canonical orientation of a point, see for example \cite{orient}.

Our assumptions allow one to define a relative Maslov index $I(x,y)
\in \zz/N\zz$ for every $x,y \in L_0 \cap L_1,$ such that $I(x,y)
\equiv I(u) \pmod N$ for any $u \in \MM(x, y).$ The relative index
satisfies $I(x,y) + I(y,z) = I(x,z),$ and induces a relative
$\zz/N\zz$-grading on the chain complex.

The Lagrangian Floer homology groups $HF(L_0,L_1)$ are the homology
groups of $CF_*(L_0, L_1)$ with respect to the differential $\del$.
Equation \eqref{break} implies that $\del^2 = 0$.  An important
property of the Floer homology groups $HF(L_0, L_1)$ is that they are
independent of the choice of path of almost complex structures, and
invariant under Hamiltonian isotopies of either $L_0$ and $L_1.$ Since
$H_1(L_0)=H_1(L_1)=0,$ any isotopy of $L_0$ or $L_1$ through
Lagrangians can be embedded in an ambient Hamiltonian isotopy; see for
example \cite[Section 6.1]{Polterovich} or the discussion in
\cite[Section 4(D)]{SeidelSmith}.

\subsection{A relative version}
\label {sec:relfloer}

 Let $\div \subset M$ denote an symplectic hypersurface disjoint from
 the Lagrangians $L_0,L_1$.  Each pseudo-holomorphic strip $u: \R
 \times [0,1] \to M$ meeting $\div$ in a finite number of points has a
 well-defined intersection number $u\cdot \div$, defined by a signed
 count of intersection points of generic perturbations.  The
 intersection numbers $u \cdot \div$ depend only on the relative
 homology class of $u$, and are additive under concatenation of
 trajectories:
$$ (u \# v)\cdot \div = (u\cdot \div) + (v \cdot \div) .$$

Let $\J(M,\omega,\div)$ denote the space of compatible almost complex
structures $J$ for which $\div$ is a $J$-holomorphic submanifold. 
Let also $\jt(M,\omega, \div)=C^{\infty}([0,1], \J(M, \omega, \div))$
be the corresponding space of time-dependent almost complex
structures. If
$(J_t) \in \jt(M,\omega,\div)$, then the intersection number of any
$J_t$-holomorphic strip with $\div$ is a finite sum of positive local
intersection numbers, see for example Cieliebak-Mohnke
\cite[Proposition 7.1]{cm:tr}. 
In particular, if a $J_t$-holomorphic strip has trivial intersection number with $R$, it must be disjoint from $R.$

One can show that $\J_t^{\reg}(L_0,L_1,\div) =\jt^{\reg}(L_0,L_1) \cap
\jt(M,\omega,R)$ is comeagre in $\jt(M,\omega,R)$. Since $L_0,L_1$ are disjoint from
$\div$, Floer homology may be defined using $J_t \in \jt(M,\omega,\div)$.
Moreover, for $J \in
\J_t^{\reg}(L_0,L_1,\div)$ the Floer differential decomposes as the sum
$$ \partial = \sum_{m \ge 0} \partial_m $$
where $\partial_m$ counts the trajectories with intersection number
$m$ with $\div$.  By additivity of the intersection numbers, the
square of the Floer differential satisfies the refined equality
$$ \sum_{i + j = m} \partial_i \partial_j = 0 .$$
In particular, $\partial_0^2 = 0$.  Let $HF(L_0,L_1;\div)$ denote the 
homology of $\partial_0$, counting Floer trajectories disjoint from
$\div$. We call $HF(L_0, L_1; R)$ the {\em Lagrangian Floer homology of $L_0,
  L_1$ relative to the hypersurface $R.$} This kind of construction
has previously appeared in the literature in various guises; see for
example Seidel's deformation of the Fukaya category
\cite[p.8]{VanCycles} or the hat version of Heegaard Floer homology
\cite{HolDisk}.  Note that $HF(L_0, L_1; R)$ admits a relative
$\zz/N'\zz$-grading, where $N' = 2N_{M\setminus R}$ is a positive
multiple of $N.$

The standard continuation argument then shows that $HF(L_0,L_1;\div)$
is independent of the choice of $J_t \in \J_t^{\reg}(L_0,L_1,\div)$.
Indeed, any two such compatible almost complex structures can be
joined by a path $J_{t,\rho}, \rho \in [0,1]$, which equips the
fiber-bundle $\R \times [0,1] \times M \to \R \times [0,1]$ with an
almost complex structure.  The part of the continuation map counting
pseudoholomorphic sections with zero intersection number with the
almost complex submanifold $\R \times [0,1] \times \div$ defines an
isomorphism from the two Floer homology groups.

In fact, we may assume that all Floer trajectories are transverse to
$\div$ by the following argument, which holds for
not-necessarily-monotone $M$. 

For any $k \in \NN$, we denote by
$\M(x_+,x_-;k)$ the subset of $\M(x_+,x_-)$ with a tangency of order
exactly $k$ to $\div$.  Given an open subset $\opens \subset M $
containing $L_0$ and $L_1$ with closure disjoint from $\div$ and a
$\ti{J} \in \J(M,\omega,\div)$, we denote by
$\jt(M,\omega,\opens,\ti{J})$ the space of compatible almost complex
structures that agree with $\ti{J}$ outside $\opens$.

\begin{lemma}  
\label{finite}
There exists a comeagre subset $\jt^\reg(L_0,L_1,\opens,\ti{J})$ of
$\jt(M,\omega,\opens,\ti{J})$ contained in $\jt^{\reg}(L_0,L_1,R)$
such that for any $(J_t) \in \jt^\reg(L_0,L_1,\opens,\ti{J}),$ the corresponding moduli space $\M(x_+,x_-)$ is a smooth manifold, and for every $k \in \NN$
and $x_\pm \in L_0 \cap L_1$, $\M(x_+,x_-;k)$ is a smooth submanifold
of $\M(x_+,x_-)$ of codimension $2k$.
\end{lemma}

\begin{proof} For the closed case, 
see Cieliebak-Mohnke \cite[Proposition 6.9]{cm:tr}.  The
proof for Floer trajectories and compatible almost complex structures
is the same, since the Lagrangians are disjoint from $\div$. Note that \cite{cm:tr} uses tamed almost complex structures; however, the arguments apply equally well to compatible almost complex structures, see \cite[p.47]{JHolCurves}.
\end{proof}

\begin{corollary}  \label {cor:transv} If $(J_t) \in \jt^{\reg}(L_0,L_1,\opens,\ti{J})$ then 
for every element of $\M(x_+,x_-)_0$ and $\M(x_+,x_-)_1$, the
intersection with $\div$ is transversal and the number of intersection
points equals the intersection pairing with $\div$.
\end{corollary} 

\subsection {Floer homology on semipositive manifolds}
\label {sec:semi}

In this section we extend the definition of Floer homology to a
semipositive setting.  More precisely, we assume the following:

\begin{assumption} \label{assumptions}
\begin {enumerate}
\item { $(M, \oe)$ is a compact symplectic manifold;}
\item  {$\oo$ is a closed two-form on $M$;}
\item {The degeneracy locus $\div \subset M$ of $\oo$ is a symplectic hypersurface with respect to $\oe$;}
\item {$\oo$ is monotone, i.e. $[\oo] = \kappa \cdot c_1(TM)$ for some
  $\kappa > 0;$}
\item {The restrictions of $\oo$ and $\oe$ to $M \setminus \div$ have
  the same cohomology class in $H^2(M \setminus \div);$}
\item {The forms $\oo$ and $\oe$ themselves coincide on an open subset
  $\opens \subset M \setminus \div;$}
\item {We are given two closed submanifolds $L_0, L_1 \subset \opens$
  which are Lagrangian with respect to $\oe$ (hence Lagrangians with
  respect to $\ti{\omega}$ as well);}

\item {$L_0$ and $L_1$ intersect transversely;}
\item {$ \pi_1(L_0) = \pi_1(L_1) =1$ and $w_2(L_0) = w_2(L_1) = 0$;}
\item {The minimal Chern number $N_{M \setminus R}$ (with respect to
  $\oe$) is at least $2,$ so that $N = 2N_{M \setminus R} \geq 4$;}
\item {There exists an almost complex structures that is compatible with
  respect to $\oe$ on $M$, and compatible with respect to $\oo$ on $M
  \setminus \div,$ and for which $\div$ is an almost complex
  submanifold. We fix such a $\tilde J,$ which we call the {\em base almost complex
  structure}.}
\item Any $\tilde J$-holomorphic sphere in $M$ of index zero
  (necessarily contained in $\div$) has intersection number with
  $\div$ equal to a negative multiple of $2$.
\end {enumerate}
\end{assumption} 

Let us remark that, because $\tilde J$ is compatible with respect to $\oo$
on $M \setminus \div,$ by continuity it follows that $\tilde J$ is
semipositive with respect to $\oo$ on all of $M$; i.e., $\oo(v, \tilde
J v) \geq 0$ for any $m \in M$ and $v\in T_mM$.

Our goal is to define a relatively $\zz/N\zz$-graded Floer homology
group $HF(L_0, L_1, \tilde J; \div)$ using Floer trajectories away
from $\div$ and a path of almost complex structures that are small
perturbations of $\tilde J$ supported in a neighborhood of $L_0 \cup
L_1$. The construction is similar to the one in
Section~\ref{sec:relfloer}, but a priori it depends on $\tilde J.$

\begin {definition}

(a) We say that $J \in \J(M,\oe)$ is {\em 
spherically
  semipositive} if every $J$-holomorphic sphere 
has non-negative Chern number $c_1(TM)[u] \geq 0$.

(b) We say that $J
\in \J(M,\oe)$ is {\em 
hemispherically semipositive} if $J$ is 
spherically semipositive and every $J$-holomorphic map
$(D^2,\partial D^2) \to (M,L_i), \ i\in{0,1}$ 
has non-negative Maslov index $I(u)$; and, further, if $I(u) =0$ then $u$ is constant.  
\end {definition}

Given a continuous map $u: (D^2, \partial D^2) \to (M,L_i), i=0,1$, we define the
{\em canonical area} of $u$ by
$$ \ti{A}(u) := \frac{ [\oo](u)}{\kappa}. $$

\begin {lemma} 
\label {areaindex}
We have $I(u) = \ti{A}(u),$ for any $u: (D^2,\partial D^2) \to (M,L_i)$. 
\end {lemma}

\begin {proof}
Since $L_i$ is simply connected, we can find a disk $v$ contained in $L_i$ with boundary equal to that of $u$, but with reversed orientation. Let $u
\# v: S^2 \to M$ the map formed by gluing.  By additivity of Maslov
index 
$$I(u) = I(u) + I(v) = I(u \# v) = 2\frac{ [\oo] ( u \# v)}{\kappa} =
%
%C:
2 \frac{[\oo](u)}{\kappa},$$ 
since both the index and the area of $v$ are trivial.
\end {proof} 
 
We define a {\em strip with decay near the ends} to be a continuous map 
\begin{equation} \label{strip}
u: (\R \times [0,1], \R \times \{0 \}, \R \times \{1 \}) \to
(M,L_0,L_1) \end{equation}
such that $\lim_{s \to \infty} u(s,t), \lim_{s \to -\infty} u(s, t)
\in L_0 \cap L_1$ exist. Every strip with decay near the ends admits a
relative homology class in $H_2(M, L_0 \cup L_1),$ and therefore has a
well-defined canonical area 
$$ \ti{A}(u) := \frac{ [\oo](u) }{\kappa}$$
and a Maslov index $I(u).$

The following lemma is \cite[Proposition 2.7]{OhMonotone}:
\begin {lemma}
\label {strips}
 Strips \eqref{strip} satisfy an index-action relation $$I(u) =
 \tilde{A}(u) + C,$$ for some constant $C$ depending only on the
 endpoints of $u.$
\end {lemma}

\begin {proof}[Proof (sketch)] Pick $u_0$ a reference strip with the same endpoints as $u.$ Using the fact that $\pi_1(L_0) =1,$ we can find a map $v:D^2 \to L_0$ such that half of its boundary is taken to the image of $u_0(\rr \times \{0\})$ and the other half to the image of $u(\rr \times \{0\}).$ By adjoining $v$ to $u$ and $u_0$ (the latter taken with reversed orientation), we obtain a disk $(-u_0) \# v \# u$ with boundary in $L_1.$ Applying Lemma~\ref{areaindex} to this disk, and using the additivity of the index and canonical action under gluing, we obtain
$$ I(u) - I(u_0) = \tilde{A} (u) - \tilde{A}(u_0).$$ We then take $C =
  I(u_0) - \tilde{A}(u_0).$ \end {proof}
As in Section \ref{sec:relfloer}, $\J(M,\omega,\opens,\ti{J})$ denotes
the space of compatible almost complex structures agreeing with
$\ti{J}$ outside $\opens$. We let $\jt(M,\omega,\opens,\ti{J}) = C^\infty( [0,1], \J(M,\omega,\opens,\ti{J}).$

\begin {lemma}  \label{openac}
Every $J$ in $\J(M,\omega,\opens,\ti{J})$ is hemispherically
semipositive.
\end {lemma}

\begin {proof}   
Since $\oo$ agrees with $\oe$ on $\opens$, 
we have $\oo(v,J v) \geq 0$ for every $v \in T_mM,$ where $m \in \opens.$ Since $J$ agrees with $\ti{J}$ outside $\opens,$ we in fact have $\oo(v,J v) \geq 0$ everywhere.  Nonnegativity of $I$ then
  follows from the monotonicity of $\oo$ (for spheres) and
  Lemma~\ref{areaindex} for disks. If a $J$-holomorphic disk $u$ has
  $I(u) = 0,$ its canonical area must be zero. Since $ J$ is
  compatible with respect to $\oo$ on $M \setminus \div,$ the disk
  should be contained in $\div.$ However, this is impossible, because
  the disk has boundary on a Lagrangian $L_i$ with $L_i \cap \div =
  \emptyset.$ (By contrast, we could have $I(u) = 0$ for
  non-constant $\tilde J$-holomorphic spheres contained in $\div.$)
\end{proof}  

Let $\jt^{\reg}(L_0,L_1,\opens,\ti{J}) \subset \jt(M,\omega,\opens,\ti{J}) \cap \jt^{\reg}(L_0,L_1,R)$ be as in Lemma \ref{finite}.
  
\begin{proposition} 
\label {prop:semifloer}
Let $M, L_0, L_1, \oo, \oe, \tilde J$ satisfy Assumption~\ref{assumptions}.  If we choose $(J_t) \in
\jt^{\reg}(L_0,L_1,$ $\opens,\ti{J}),$ 
then the relative Floer differential
counting trajectories disjoint from $\div$ is finite and satisfies
$\partial_0^2 = 0$. The resulting (relatively $\zz/N\zz$-graded) Floer
homology groups $HF_*(L_0,L_1,\tilde J;\div)$ are independent of the choice of path $(J_t)$, and are preserved under Hamiltonian isotopies of either Lagrangian, as long as
Assumption~\ref{assumptions} is still satisfied.
\end{proposition}

\begin{proof}   Using parts (v) and (vi) of Assumption~\ref{assumptions}, we see that on the complement of $\div$ we have $\oe - \oo = da,$ for some $a\in \Omega^1(M \setminus \div)$ satisfying $da=0$ on the neighborhood $\opens$ of $L_0 \cup L_1.$ Let $u$ be a pseudo-holomorphic strip whose image is contained in $M \setminus \div$. Then 
$$E(u) - \kappa \tilde{A} (u) = \int_{\rr \times [0,1]} u^*(\oe - \oo)
  = \int_{\rr \times [0,1]} d(u^*a) = \int_{\gamma_0} u^*a -
  \int_{\gamma_1} u^*a,$$
where $\gamma_i$ is a path in the Lagrangian $L_i$ joining the
endpoints of $u.$ Since $da=0$ on $L_i,$ Stokes' Theorem implies that
$\int_{\gamma} u^*a $ is independent of $\gamma$; it just depends on
the endpoints. Therefore, $ E(u) - \kappa \tilde{A} (u) $ only depends
on the endpoints of $u.$ Together with Lemma~\ref{strips} this gives
an energy index relation as follows: for any $u$ in $M\setminus \div,$
we have
$$ I(u) = E(u)/\kappa + C',$$
where $C'$ is a constant depending on the endpoints of $u.$ Since
there is a finite number of possibilities for these endpoints, it
follows that there exists a constant $K > 0$ such that the energy of
any such trajectory $u$ is bounded above by $K$.

Let $(J_t) \in \jt^\reg(L_0,L_1,\opens,\ti{J})$.  By
Proposition~\ref{openac} each $J_t$ is 
hemispherically semipositive.  We define the Floer differential by
counting $J_t$-holomorphic strips in $M \setminus \div.$ By Lemma
\ref{finite}, a sequence of such strips cannot converge to a strip
that intersects $\div,$ unless further bubbling occurs.

We seek to rule out sphere bubbles and disk bubbles in the boundary of
the zero and one-dimensional moduli spaces of such strips (i.e. those
of index $1$ or $2$).  Assume that we have a sequence $(u_\nu)$ of
pseudo-holomorphic strips of index $1$ or $2.$ 
Because of the energy bound, a subsequence Gromov
converges to a limiting configuration consisting of a broken
trajectory and a collection of disk and sphere bubbles.  Since the
$J_t$'s are hemispherically semipositive, it follows that the indices of
the bubbles are nonnegative. Further, by part (x) of Assumption~\ref{assumptions}, the index of each bubble is a multiple of $4.$ Since we started with a configuration of index at
most $2$, all bubbles have index zero.  By the definition of hemispherical semipositivity, the index zero
disks are constant.  

By item (xii) of Assumption \ref{assumptions}, each index sphere
bubble contributes a multiple of two to the intersection number with
$\div$.  By Lemma \ref{finite}, the intersection number of the
limiting trajectory $u_\infty$ (with sphere bubbles removed) is given by the
number of intersection points, and each of these is transverse.  Hence
at most half of the intersection points with $\div$ have sphere
bubbles attached.  In particular, there exists a point $z \in \R
\times [0,1]$ such that $u_\infty(z) \in \div$ is a transverse intersection
point but $z$ is not in the bubbling set.  Since the intersection
points are stable under perturbation, it follows that $u_\infty$
cannot be a limit of Floer trajectories disjoint from $\div$. Indeed, by
definition this convergence is uniform in all derivatives on the
complement of the bubbling set and, in particular, on an open subset
containing $z$.  Since there are no sphere bubbles, there cannot be
any disk bubbles either, since at least one disk bubble would have to
be non-constant.  Hence the limit is a (possibly broken) Floer
trajectory. 

 The rest of the argument is then as in the monotone case. In particular, the statement about the invariance of $HF(L_0, L_1, \tilde J; R)$ follows from the usual continuation arguments in Floer theory.
\end{proof} 

\begin {remark}
If $M, L_0, L_1, \oo, \oe, \tilde J$ satisfy Assumption~\ref{assumptions}, we can define $HF_*(L_0,L_1, \tilde J;\zed)$ even if $L_0$ and $L_1$ do not intersect transversely: one can simply isotope one of the Lagrangians to achieve transversality, and take the resulting Floer homology. 
\end {remark}

\begin {remark}
A priori the construction of the Floer homologies $HF(L_0, L_1, \tilde
J; \zed)$ depends on the open set $\opens,$ because $(J_t)$ is chosen
from the corresponding set $\jt^{\reg}(L_0,L_1,\opens,\ti{J}).$
However, suppose we have another open set $\opens' \subset M \setminus
R$ satisfying $L_0 \cap L_1 \subset \opens'$ and $\oe = \oo$ on
$\opens'.$ Note that
\begin {equation}
\label {eq:ww}
\jt^{\reg}(L_0,L_1,\opens,\ti{J}) \cap
\jt^{\reg}(L_0,L_1,\opens',\ti{J}) = \jt^{\reg}(L_0,L_1,\opens \cap
\opens', \ti{J}), \end {equation} because the regularity condition in
Lemma~\ref{finite} is intrinsic for $(J_t)$ (it boils down to the
surjectivity of certain linear operators). It follows that by choosing
$(J_t)$ in the (necessarily nonempty) intersection \eqref{eq:ww}, the
Floer homologies $HF(L_0, L_1, \tilde J; \zed)$ defined from $\opens$
and $\opens'$ are isomorphic. Thus, we can safely drop $\opens$ from
the notation.
 \end {remark}

\begin {remark}
\label {rem:jt}
A smooth variation of the base almost complex structure $\tilde J$ induces an isomorphism between the respective Floer homologies $HF(L_0, L_1, \tilde J; R)$. However, if we are only given $\oo$ and $\oe,$ it is not clear whether the space of possible $\tilde J$'s is contractible. This justifies keeping $\tilde J$ in the notation $HF(L_0, L_1, \tilde J; R)$.
\end {remark}

\section {Moduli spaces}
\label {sec:moduli}

\subsection {Notation} 
\label {sec:notation}

Throughout the rest of the paper $G$ will denote the Lie group
$SU(2)$, and $\gad = PU(2) = SO(3)$ the corresponding group of
adjoint type. We identify the Lie algebra $\lieg = \mathfrak{su}(2)$
with its dual $\lieg^*$ by using the basic invariant bilinear form
$$ \langle \cdot , \cdot \rangle: \lieg \times \lieg \to \rr,
\ \ \langle A, B\rangle = - \tr(AB). $$

The maximal torus $T \cong S^1 \subset G$ consists of the diagonal
matrices $\diag(e^{2\pi t i}, e^{-2\pi ti}), t \in \rr.$ We let $\tad = T/(\zz/2\zz) \subset \gad$ and identify their Lie
algebra $\tt$ with $\rr$ by sending 
$\diag(i, -i)$ to $1.$ Under this identification, the restriction of
the inner product $\langle \cdot , \cdot \rangle$ to $\tt$ is twice
the Euclidean metric. We use this inner product to identify $\tt$ with
$\tt^*$ as well.  Finally, we let $\tt^\perp$ denote the
orthocomplement of $\tt$ in $\lieg.$

Conjugacy classes in $\lieg$ (under the adjoint action of $G$) are
parametrized by the positive Weyl chamber $\tt_+ = [0, \infty).$
  Indeed, the adjoint quotient map
$$ Q : \lieg \to [0, \infty)$$ takes $\theta \in \lieg$ to $t$ such
    that $\theta$ is conjugate to $\diag(ti, -ti).$

On the other hand, conjugacy classes in $G$ are parametrized by the
fundamental alcove $\alcove = [0,1/2].$ Indeed, for any $g \in G,$
there is a unique $t \in [0, 1/2]$ such that $g$ is conjugate to the
diagonal matrix $\diag(e^{2\pi ti}, e^{- 2\pi t i}).$

\subsection {The extended moduli space} \label{sec:ext} 
We review here the construction of the \emph{extended moduli space}
(\cite{Jeffrey}, \cite {Huebschmann}), mostly following Jeffrey's
gauge-theoretic approach from \cite{Jeffrey}.

Let $\Sigma$ be a compact connected 
Riemann surface of genus $h \geq 1.$ Fix some $z \in \Sigma$ and let
$\Sigma'$ denote the complement in $\Sigma$ of a small disk around
$z,$ so that $S= \del \Sigma'$ is a circle. Identify a neighborhood of
$S$ in $\Sigma'$ with $[0, \eps) \times S,$ and let $s \in
  \rr/2\pi\zz$ be the coordinate on the circle $S$.

Consider the space $\aa(\Sigma') \cong \Omega^1(\Sigma') \otimes
\lieg$ of smooth connections on the trivial $G$-bundle over $\Sigma'$,
and set
$$ \aa^{\lieg}(\Sigma') = \{A \in \aa(\Sigma') \ | \ F_A = 0, \ A = \theta\diff s \text{ on some neighborhood of $S$ for some } \theta \in \lieg \}.$$

The space $ \aa^{\lieg}(\Sigma')$ is acted on by the gauge group
$$ \gauge^c(\Sigma') = \{f :\Sigma' \to G) \ |\ f = I \text{ on some
  neighborhood of } S \}.$$

The extended moduli space is then defined as
$$ \mm^{\ext}(\Sigma') = \aa^{\lieg}(\Sigma') / \gauge^c(\Sigma').$$

A more explicit description of the extended moduli space is obtained
by fixing a collection of simple closed curves $\alpha_i, \beta_i
\ (i=1, \dots, h)$ on $\Sigma',$ based at a point in $S,$ such that
$\pi_1(\Sigma')$ is generated by their equivalence classes and the
class of a curve $\gamma$ around $S$, with the relation:
$\prod_{i=1}^h [\alpha_i, \beta_j] = \gamma.$

To each connection on $\Sigma'$ one can then associate the holonomies
$A_i, B_i \in G$ around the loops $\alpha_i$ and $\beta_i,$
respectively, $i=1, \dots, h.$ This allows us to view the extended
moduli space as
\begin {equation}
\label {eq:comm}
 \mm^{\ext}(\Sigma') = \{(A_1, B_1, \dots, A_h, B_h) \in G^{2h}, \theta \in \lieg \ | \ \prod_{i=1}^h [A_i, B_i] = \exp (2\pi\theta) \}.
\end {equation}

There is a proper map $$\Phi: \mm^{\ext}(\Sigma') \to \lieg$$ which
takes the class $[A]$ of a connection $A$ to the value $\theta =
\Phi(A)$ such that $A|_S = \theta ds.$ (This corresponds to the
variable $\theta$ appearing in \eqref{eq:comm}.) There is also a
natural $G$-action on $ \mm^{\ext}(\Sigma')$ given by constant gauge
transformations. With respect to the identification \eqref{eq:comm},
it is
\begin {equation}
\label {eq:gaction}
 g \in G \ : (A_i, B_i, \theta) \to (gA_ig^{-1}, gB_ig^{-1}, \Ad(g)\theta).
 \end {equation}

Observe that this action factors through $\gad.$ The map $\Phi$ is
equivariant with respect to this action on its domain, and the adjoint
action on its target. Set
$$ \tilde \Phi : \mm^\ext(\Sigma') \to [0, \infty), \ \ \ \tilde \Phi = Q \circ \Phi.$$

Now consider the subspace
$$ \mm^{\ext}_s(\Sigma') = \{ x \in \mm^\ext(\Sigma') \ | \ \tilde
\Phi(x) \not \in \zz \setminus \{0\} \ \}. $$

\begin {proposition}
\label {smoothness}
(a) The space $\mm^{\ext}_s(\Sigma')$ is a smooth manifold of real dimension  $6h$. 

(b) Every nonzero element $\theta \in \lieg$ is a regular value for
the restriction of $\Phi$ to $\mm^{\ext}_s(\Sigma').$
\end {proposition}

\begin {proof}
Part (a) is proved in \cite[Theorem 2.7]{Jeffrey}. We copy the proof
here, and explain how the same arguments can be used to deduce part
(b) as well.
 
Consider the commutator map $c: G^{2h} \to G, c(A_1, B_1, \dots, A_h,
B_h)= \prod_{i=1}^h [A_i, B_i].$ For $\rho = (A_1, B_1, \dots, A_h,
B_h) \in G^{2h},$ we denote by $Z(\rho) \subset G$ its stabilizer
(under the diagonal action by conjugation). Let $z(\rho) \subset
\lieg$ be the Lie algebra of $Z(\rho).$ Note that $Z(\rho) = \{\pm 
I\}$ unless 
$c(\rho) = I.$

The image of $dc_\rho \cdot c(\rho)^{-1}$ is $z(\rho)^\perp;$ see for
example \cite[proof of Proposition 3.7]{Goldman}. In particular, the
differential $dc_\rho$ is surjective whenever $c(\rho) \neq I.$

Define the maps
$$ f_1: G^{2h} \times \lieg \to G, \ \ f_1(\rho, \theta) = c(\rho)
\cdot \exp (-2\pi\theta)$$ and
$$ f_2: G^{2h} \times \lieg \to G \times \lieg, \ \ f_2(\rho, \theta)
= (f_1(\rho, \theta), \theta).$$

On the extended moduli space $\mm^\ext(\Sigma') = f_1^{-1}(I),$ we have
$$ (df_1)_{(\rho, \theta)} = (dc)_\rho \exp(-2\pi\theta) + 2\pi
\exp(2\pi\theta)(d\exp)_{-2\pi\theta}.$$

When $c(\rho) = \exp(2\pi\theta) \neq I,$ we have that $(dc)_\rho$ is
surjective, hence so is $(df_1)_{(\rho, \theta)}.$ Also, when $\theta
= 0,$ $(d\exp)_{-2\pi\theta}$ is just the identity, so again
$(df_1)_{(\rho, \theta)}$ is surjective. Claim (a) follows.

Next, observe that
$$ (df_2)_{(\rho, \theta)} (\alpha, \lambda) = ((df_1)_{(\rho,
  \theta)} (\alpha, \lambda), \lambda) = (dc_{\rho}(\alpha) \cdot
\exp(-2\pi\theta) + l(\lambda), \lambda),$$ 
where $l(\lambda)$ does not depend on $\alpha.$ Hence, when $c(\rho)=
\exp(2\pi \theta) \neq I,$ the differential $(df_2)_{(\rho, \theta)}$
is surjective. This implies that any $\theta \in \lieg$ with
$Q(\theta) \not \in \zz$ is a regular value for
$\Phi|_{\mm^{\ext}_s(\Sigma')}.$ Since the values $\theta \in \lieg$
with $Q(\theta) \in \zz \setminus \{0\}$ are not in the image of
$\Phi|_{\mm^{\ext}_s(\Sigma')}$, they are automatically regular
values, and claim (b) follows.
\end {proof}

Consider also the subspace
$$\nn(\Sigma') = \tilde \Phi^{-1}\bigl( [0, 1/2) \bigr) \subset
  \mm_s^{\ext}(\Sigma').$$

Note that the restriction of the exponential map $\theta \to \exp(2\pi
\theta)$ to $Q^{-1}\bigl[0, 1/2\bigr)$ is a diffeomorphism onto its
  image $G \setminus \{-I\}.$ Therefore, using the identification
  \eqref{eq:comm}, we can describe $\nn(\Sigma')$ as
\begin {equation}
\label {nsigma}
 \nn (\Sigma') = \left\{(A_1, B_1, \dots, A_h, B_h) \in G^{2h}\ \big|
 \ \prod_{i=1}^h [A_i, B_i] \neq -I \right\}.
\end {equation}

\subsection {Hamiltonian actions.} \label {sec:red}
Let $K$ be a compact, connected Lie group with Lie algebra $\liek.$ We
let $K$ act on the dual Lie algebra $\liek^*$ by the coadjoint action.

A {\em pre-symplectic manifold} is a smooth manifold $M$ together with
a closed form $\omega \in \Omega^2(M)$, possibly degenerate.  A {\em Hamiltonian pre-symplectic $K$-manifold} $(M, \omega, \Phi)$ is a pre-symplectic manifold $(M, \omega)$ together
with a $K$-equivariant smooth map $\Phi: M \to \liek^*$ such that for
any $\xi \in \lieg,$ if $X_{\xi}$ denotes the vector field on $M$
generated by the one-parameter subgroup $\{ \exp (-t\xi) | t \in
\rr\} \subset K$, we have
$$ d\bigl( \langle \Phi, \xi \rangle \bigr) = - \iota(X_{\xi})
\omega.$$

Under these hypotheses, the $K$-action on $M$ is called Hamiltonian,
and $\Phi$ is called the {\em moment map}. The quotient
$$ M /\!\!/K \ := \ \Phi^{-1}(0)/K $$ is named the {\em pre-symplectic
  quotient} of $M$ by $K.$ The following result is known as the
Reduction Theorem (\cite{MarsdenWeinstein}, \cite{Meyer},
\cite[Theorem 5.1]{GGK}):

\begin {theorem}
\label {thm:pre}
Let $(M, \omega, \Phi)$ be a Hamiltonian pre-symplectic
$K$-manifold. Suppose that the level set $\Phi^{-1}(0)$ is a smooth
manifold on which $K$ acts freely. Let $i:\Phi^{-1}(0) \hookrightarrow
M$ be the inclusion and $\pi: \Phi^{-1}(0) \to M /\!\!/K$ the
projection. Then there exists a unique closed form $\omega_{\red}$ on
the smooth manifold $ M /\!\!/K$ with the property that $i^* \omega =
\pi^* \omega_{\red}.$ The reduced form $\omega_{\red}$ is
non-degenerate on $ M /\!\!/K$ if and only if $\omega$ is
nondegenerate on $M$ at the points of $\Phi^{-1}(0).$

Furthermore, if $M$ admits another Hamiltonian $K'$-action (for some
compact Lie group $K'$) that commutes with the $K$-action, then $(M
/\!\!/K, \omega_{\red})$ has an induced Hamiltonian $K'$-action.
\end {theorem}

When the form $\omega$ is symplectic, $(M, \omega, \Phi)$ is simply
called a {\em Hamiltonian $K$-manifold}. In this case we can drop the
condition that $\Phi^{-1}(0)$ is smooth from the hypotheses of
Theorem~\ref{thm:pre}; indeed, this condition is automatically implied
by the assumption that $K$ acts freely on $\Phi^{-1}(0).$

\subsection {A closed two-form on the extended moduli space}
\label {sec:symp}

According to \cite[Equation (2.7)]{Jeffrey}, the tangent space to the
smooth stratum $\mm_s^{\ext}(\Sigma') \subset \mm^{\ext}(\Sigma')$ at
some class $[A]$ can be naturally identified with
\begin {equation}
\label {tangent}
 T_{[A]} \mm_s^{\ext}(\Sigma') = \frac{\ker(\diff_A : \Omega^{1,
     \lieg}(\Sigma') \to \Omega^2_c(\Sigma') \otimes \lieg)}{\im
   (\diff_A: \Omega^0_c(\Sigma')\otimes \lieg \to \Omega^{1,
     \lieg}(\Sigma'))},
\end {equation}
where $\Omega^p_c(\Sigma')$ denotes the space of $p$-forms compactly supported in the interior of $\Sigma',$ and $\Omega^{1, \lieg}(\Sigma')$ denotes the space of 1-forms $A$ such that $A=\theta ds$ near $S =\del \Sigma'$ for some $\theta \in \lieg.$

Define a bilinear form $\omega$ on $\Omega^{1, \lieg}(\Sigma')$ by 
$$\omega(a, b) = \int_{\Sigma'} \tr (a \wedge b),$$ where the wedge operation on $\lieg$-valued forms combines the usual exterior product with the inner product on $\lieg.$ Stokes' Theorem implies that $\omega$ descends to a bilinear form on the tangent space to $\mm^{\ext}_s(\Sigma')$ described in Equation \eqref{tangent} above. Thus we can think of $\omega$ as a two-form on $\mm^{\ext}_s(\Sigma').$

\begin {theorem}[Huebschmann-Jeffrey]
\label {jnondeg}
The two-form $\omega \in \Omega^2(\mm^{\ext}_s(\Sigma'))$ is
closed. It is nondegenerate when restricted to $\nn (\Sigma') \subset
\mm_s^{\ext}(\Sigma').$ Moreover, the restriction of the $\gad$-action
\eqref{eq:gaction} to $\mm^{\ext}_s(\Sigma')$ is Hamiltonian with
respect to $\omega.$ Its moment map is the restriction of $\Phi$ to
$\nn(\Sigma'),$ which we henceforth also denote by $\Phi.$
\end {theorem}

For the proof, we refer to Jeffrey \cite{Jeffrey}; see also
\cite{MeinWoodVerlinde}.

Theorem~\ref{jnondeg} says that $(\mm^{\ext}_s(\Sigma'), \omega,
\Phi)$ is a Hamiltonian pre-symplectic $\gad$-manifold in the sense of
Section~\ref{sec:red}, and that its subset $(\nn(\Sigma'),\omega,
\Phi)$ is a (symplectic) Hamiltonian $\gad$-manifold. The symplectic
quotient
$$ \nn(\Sigma')/\!\!/ \gad = \Phi^{-1}(0)/\gad = \mm(\Sigma)$$
is the usual moduli space of flat $G$-connections on $\Sigma,$ with
the symplectic form (on its smooth stratum) being the one constructed
by Atiyah and Bott \cite{AtiyahBott}. If $\Sigma$ is given a complex
structure, $\mm(\Sigma)$ can also be viewed as the moduli space of
semistable bundles of rank two on $\Sigma$ with trivial determinant,
cf. \cite{NarasimhanSeshadri}.

For an alternate (group-theoretic) description of the form $\omega$ on
$\nn(\Sigma')$, see \cite{Jeffrey2}, \cite{Huebschmann}, or
\cite{HuebJeffrey}.

Let us mention two results about the two-form $\omega.$ The first is
proved in \cite{MeinWoodCanonical}:
\begin {theorem}[Meinrenken-Woodward]
\label {thm:mw}
$(\nn(\Sigma'), \omega)$ is a monotone symplectic manifold, with
monotonicity constant $1/4.$
\end {theorem}

The second result is:
\begin {lemma}
\label {lemma:z}
The cohomology class of the symplectic form $\omega \in
\Omega^2(\nn(\Sigma'))$ is integral.
\end {lemma}

\begin {proof}
The extended moduli space $\mm^\ext(\Sigma')$ embeds in the moduli
space $\mm(\Sigma')$ of all flat connections on $\Sigma'.$ The latter
is an infinite dimensional Banach manifold with a natural symplectic
form that restricts to $\omega$ on $\mm^\lieg_s(\Sigma').$ Moreover,
Donaldson \cite{DonaldsonB} showed that $\mm(\Sigma')$ has the
structure of a Hamiltonian $LG$-manifold, where $LG=
\operatorname{Map}(S^1, G)$ is the loop group of $G.$

Recall that a pre-quantum line bundle $E$ for a symplectic manifold
$(M, \omega)$ is a Hermitian line bundle equipped with an invariant
connection $\nabla$ whose curvature is $-2\pi i$ times the symplectic
form. If $M$ is finite dimensional, this implies that $[\omega]=c_1(E)
\in H^2(M; \zz).$ In our situation, a pre-quantum line bundle on $M =
\nn(\Sigma')$ can be obtained by restricting the well-known
$LG$-equivariant pre-quantum line bundle on the infinite-dimensional
symplectic manifold $\mm(\Sigma').$ We refer the reader to
\cite{Mickelsson}, \cite{RamadasSingerWeitsman} and \cite{Witten2d}
for the construction of the latter; see also \cite{MeinWoodVerlinde}.
\end {proof}

\begin {corollary}
\label {cor:4}
The minimal Chern number of the symplectic manifold $\nn(\Sigma')$ is a positive multiple of $4.$
\end {corollary}

\begin {proof}
Use Theorem~\ref{thm:mw} and Lemma~\ref{lemma:z}.
\end {proof}

\subsection {Other versions} 
\label {sec:other}
Although our main interest lies in the extended moduli space
$\mm^\lieg(\Sigma')$ and its open subset $\nn(\Sigma'),$ in order to
understand them better we need to introduce two other moduli
spaces. Both of them appeared in \cite{Jeffrey}, where their main
properties are spelled out. An alternative viewpoint on them is given
in \cite[Section 3.4.2]{MeinWoodVerlinde}, where they are interpreted
as cross-sections of the full moduli space $\mm(\Sigma').$

The first auxiliary space that we consider is the {\em toroidal
  extended moduli space}:
$$ \mm^{\tt}(\Sigma') = \Phi^{-1}(\tt) \subset \mm^{\lieg}(\Sigma').$$

It has a smooth stratum
$$ \mm^{\tt}_s(\Sigma') = \{ x \in \mm^{\tt}(\Sigma') \ | \ \tilde
\Phi(x) \not \in \zz \}.$$ 
The restrictions of $\omega$ and $\Phi$ to $\mm^{\tt}_s(\Sigma')$ turn
it into a Hamiltonian pre-symplectic $\tad$-manifold. On the open
subset $\mm^\tt(\Sigma') \cap \tilde\Phi^{-1}(0, 1/2),$ the two-form
is nondegenerate.

The second space is the {\em twisted extended moduli space} from
\cite[Section 5.3]{Jeffrey}. In terms of coordinates, it is
$$\mm^{\lieg}_{\operatorname{tw}} (\Sigma') = \Bigl \{(A_1, B_1,
\dots, A_h, B_h) \in G^{2h}, \theta \in \lieg \Big | \prod_{i=1}^h
     [A_i, B_i] = -\exp(2\pi \theta) \Bigr\}.$$

This space admits a $\gad$-action just like $\mm^\lieg(\Sigma'),$ and
a natural projection $\Phi_\tw: \mm^{\lieg}_\tw \to \lieg.$ Set
$\tilde \Phi_\tw = Q \circ \Phi_\tw.$ The smooth stratum of
$\mm^{\lieg}_\tw(\Sigma')$ is
$$\mm^{\lieg}_{\tw, s} (\Sigma') =  \Bigl\{x \in \mm^\lieg_\tw(\Sigma')\ \Big| \ \tilde \Phi_\tw(x) \not \in \zz + \frac{1}{2} \Bigr\}.$$

Furthermore, $\mm^{\lieg}_{\tw, s} (\Sigma')$ admits a natural two-form $\omega_\tw,$ which turns it into a Hamiltonian pre-symplectic $\gad$-manifold, with moment map $\Phi_\tw.$ The restriction of $\omega_{\tw}$ to the subspace 
$$\nn_\tw (\Sigma') =\tilde \Phi_\tw^{-1}\bigl ([0, 1/2) \bigr)$$ is
  nondegenerate.

Observe that the subspace $\Phi_{\tw}^{-1}(\tt) \subset
\mm^{\lieg}_{\tw}(\Sigma')$ can be identified with the toroidal
extended moduli space $\mm^\tt(\Sigma'),$ via the map 
$$(A_1, B_1,\ldots,A_h,B_h, t) \to (A_1, B_1, \ldots, A_h,B_h, 1/2 -
t).$$
This map is a diffeomorphism of the smooth
strata, and is compatible with the restrictions of the pre-symplectic
forms $\omega$ and $\omega_\tw.$

\subsection {The structure of degeneracies of $\mm^{\lieg}_s(\Sigma')$}
\label {sec:degeneracies}
Recall from Theorem~\ref{jnondeg} that the degeneracy locus of the
pre-symplectic manifold $ \mm^{\lieg}_s(\Sigma')$ is contained in the
preimage $\tilde \Phi^{-1}(1/2).$ We seek to understand the structure
of the degeneracies.

Let $\mu = \diag(i/2, -i/2).$ Note that the stabilizer $\gad$ of $\exp(2\pi\mu)=-I$ is bigger than the stabilizer $\tad = S^1$ of $\mu.$ Thus, we have an obvious diffeomorphism
$$ \tilde \Phi^{-1}(1/2) \cong \omu \times \Phi^{-1}(\mu),$$ where
$\omu$ denotes the coadjoint orbit of $\mu.$ The first factor $\omu$
is diffeomorphic to the flag variety $\gad/\tad \cong \pp^1.$ The
second factor $\Phi^{-1}(\mu)$ is smooth by
Proposition~\ref{smoothness} (b).

There is a residual $\tad$-action on the space $\Phi^{-1}(\mu).$ Thus
$\Phi^{-1}(\mu)$ is an $S^1$-bundle over
$$\mm_\mu(\Sigma') = \Phi^{-1}(\mu) /\tad.$$ 
Finally, $\mm_{\mu}(\Sigma')$ is a $\pp^1$-bundle over  
$$\mm_{-I}(\Sigma') =\Bigl\{(A_1, B_1, \dots, A_h, B_h) \in G^{2h} | \prod_{i=1}^h [A_i, B_i] = -I \Bigr\}/ \gad.$$

This last space $\mm_{-I} (\Sigma')$ can be identified with the moduli
space $\mm_{\tw}(\Sigma)$ of projectively flat connections on
$\mathfrak{E}$ with fixed central curvature, where $\mathfrak{E}$ is a
$U(2)$-bundle of odd degree over the closed surface $\Sigma = \Sigma'
\cup D^2.$ Alternatively, it is the moduli space of rank two stable
bundles on $\Sigma$ having fixed determinant of odd degree,
cf. \cite{NarasimhanSeshadri}, \cite{AtiyahBott}. It can also be
viewed as the symplectic quotient of the twisted extended moduli space
from Section~\ref{sec:other}:
$$ \mm_\tw (\Sigma) = \nn_\tw (\Sigma')/\!\!/ \gad = \Phi_\tw^{-1}(0)/\gad.$$

We have described a string of fibrations that gives a clue to the
structure of the space $\tilde \Phi^{-1}(1/2).$ Let us now reshuffle
these fibrations and view $\tilde \Phi^{-1}(1/2)$ as a $\gad$-bundle
over the space $\omu \times \mm_{-I} (\Sigma').$ Its fiberwise tangent
space (at any point) is $\lieg$, which can be decomposed as $\tt
\oplus \tt^\perp,$ with $\tt^\perp \cong \cc.$

\begin{proposition} 
\label {prop:d}
Let $x \in \tilde \Phi^{-1}(1/2) \subset \mm^{\lieg}_s(\Sigma').$ The
null space of the form $\omega$ at $x$ consists of the fiber
directions corresponding to $\tt^\perp \subset \lieg.$
\end{proposition}  

\begin {proof}
Our strategy for proving Proposition~\ref{prop:d} is to reduce it to a
similar statement for the toroidal extended moduli space
$\mm^\tt(\Sigma'),$ and then study the latter via its embedding into
the twisted extended moduli space $\mm^\lieg_\tw(\Sigma').$

First, note that by $\gad$-invariance, we can assume without loss of
generality that $\Phi(x) = \mu.$ The symplectic cross-section theorem
\cite{GuilleminSternberg} says that, near $\Phi^{-1}(\mu),$ the
two-form on $\mm^{\lieg}(\Sigma')$ is obtained from the one on
$\mm^\tt(\Sigma')= \Phi^{-1}(\tt)$ by a procedure called symplectic
induction. (Strictly speaking, symplectic induction is described in
\cite{GuilleminSternberg} for nondegenerate forms; however, it applies
to the Hamiltonian pre-symplectic case as well.) More concretely, we
have a (noncanonical) decomposition
\begin {equation}
\label {eq:decompose}
 T_x \mm^\lieg(\Sigma') = T_x \mm^\tt(\Sigma') \oplus T_{\mu}(\omu)
\end {equation}
such that $\omega|_x$ is the direct sum of its restriction to the
first summand in \eqref{eq:decompose} 
with the canonical symplectic form on the second summand. 

Recall that we are viewing $\tilde \Phi^{-1}(1/2)$ as a $G$-bundle
over $\omu \times \mm_{-I}(\Sigma').$ Its intersection with
$\mm^\tt(\Sigma')$ is $\Phi^{-1}(\mu),$ which is the part of the
$\gad$-bundle that lies over $\{\mu\} \times \mm_{-I}(\Sigma').$ The
decomposition \eqref{eq:decompose} implies that, in order to prove the
final claim about the null space of $\omega|_x,$ it suffices to show
that the null space of $\omega|_{\mm^{\tt}(\Sigma')}$ at $x$ consists
of the fiber directions corresponding to $\tt^\perp \subset \lieg.$
  
Let us use the observation in the last paragraph of
Section~\ref{sec:other}, and view $\mm^\tt(\Sigma')$ as
$\Phi_\tw^{-1}(\tt)\subset \mm^\lieg_\tw(\Sigma').$ The point $x$ now
lies in $\Phi_{\tw}^{-1}(0).$

Recall from Section~\ref{sec:other} that the two-form
$\mm^\lieg_\tw(\Sigma')$ is nondegenerate near $\Phi_{\tw}^{-1}(0).$
Further, it is easy to check that the action of $\gad$ on
$\Phi_{\tw}^{-1}(0)$ is free. This action is Hamiltonian; hence, the
quotient $\Phi_{\tw}^{-1}(0)/\gad = \mm_{-I}(\Sigma')$ is smooth, and
the reduced two-form on it is nondegenerate. Further, there is a
(noncanonical) decomposition:
\begin {equation}
\label {eq:decompose2}
T_x \mm^\lieg_\tw(\Sigma') \cong \pi^* T_{\pi(x)} \mm_{-I}(\Sigma')
\oplus \lieg \oplus \lieg^*,
\end {equation}
where $\pi: \Phi_{\tw}^{-1}(0) \to \mm_{-I}(\Sigma')$ is the quotient
map. (See for example Equation (5.6) in \cite{GGK}.) The two-form
$\omega_\tw$ at $x$ is the direct summand of the reduced form at
$\pi(x)$ and the natural pairing of the two last factors in
\eqref{eq:decompose2}.

With respect to the decomposition~\eqref{eq:decompose2}, the subspace
$T_x \mm^\tt(\Sigma') \subset T_x \mm^\lieg_\tw(\Sigma')$ corresponds
to
$$ T_x \mm^\tt(\Sigma') \cong \pi^* T_{\pi(x)} \mm_{-I}(\Sigma') \oplus \lieg \oplus \tt^*.$$

Therefore, the null space of $\omega_\tw$ on $T_x \mm^\tt(\Sigma')$ is
the null space of the restriction of the natural pairing on $\lieg
\oplus \lieg^*$ to $\lieg \oplus \tt^*.$ This is $\lieg/\tt \cong
\tt^\perp,$ as claimed.
\end {proof}

\section {Symplectic cutting}
\label {sec:cutting}

\subsection {Abelian symplectic cutting.}

We review here Lerman's definition of (abelian) symplectic cutting, following \cite{Lerman}. 
Consider a symplectic manifold $(M, \omega)$ with a Hamiltonian
$S^1$-action and moment map $\Phi:M \to \rr.$ Pick some $\lambda \in
\rr$. The diagonal $S^1$-action on the space $M \times \cc^-$ (endowed
with the standard product symplectic structure, where $\cc^-$ is $\cc$
with negative the usual area form) is Hamiltonian with respect to the
moment map
$$ \Psi: M \times \cc^- \to \rr, \ \ \ \ \Psi(m, z)= \Phi(m) + \frac{1}{2}|z|^2 - \lambda.$$

The symplectic quotient
$$ M_{\leq \lambda} \ := \ \Psi^{-1}(0)/S^1 \cong
\Phi^{-1}(\lambda)/S^1 \cup \Phi^{-1}(-\infty, \lambda)$$ is called
the {\em symplectic cut} of $M$ at $\lambda.$ If the action of $S^1$
on $\Phi^{-1}(\lambda)$ is free, then $ M_{\leq \lambda}$ is a
symplectic manifold, and it contains $\Phi^{-1}(\lambda)/S^1$ (with
its reduced form) as a symplectic hypersurface, i.e. a symplectic
submanifold of real codimension two.

\begin {remark}
\label {normal}
The normal bundle to $\Phi^{-1}(\lambda)/S^1$ in $ M_{\leq \lambda}$ is the complex line bundle whose associated circle bundle is $\Phi^{-1}(\lambda) \to \Phi^{-1}(\lambda)/S^1.$
\end {remark}

\begin {remark}
\label {localcut}
Symplectic cutting is a local construction. In particular, if $(M, \omega)$ is symplectic  and $\Phi:M \to \rr$ is a continuous map that induces a smooth Hamiltonian $S^1$-action on an open set $\uu \subset M$ containing $\Phi^{-1}(\lambda)$, then we can still define $M_{\leq \lambda}$ as the union $(M \setminus \uu) \cup \uu_{\leq \lambda}.$ 
\end {remark}

\begin {remark}
\label {extraham0}
If $M$ has an additional Hamiltonian $K$-action (for some other compact group $K$)
commuting with that of $S^1$, then $M_{\leq \lambda}$ has an induced Hamiltonian $K$-action. This follows from the similar statement for symplectic reduction, cf. Theorem~\ref{thm:pre}.
\end {remark}

\subsection {Non-abelian symplectic cutting.}
\label {sec:nonabelian}

An analog of symplectic cutting for non-abelian Hamiltonian actions
was defined in \cite{Woodward}. We explain here the case of
Hamiltonian $PU(2)$-actions, since this is all we need for our
purposes.

We keep the notation from Section~\ref{sec:notation}, with $G=SU(2)$
and $\gad = PU(2).$ Let $(M, \omega, \Phi)$ be a Hamiltonian
$\gad$-manifold. Since $\lieg$ and $\lieg^*$ are identified using the
bilinear form, from now on we will view the moment map $\Phi$ as
taking values in $\lieg.$ Recall that
$$ Q: \lieg \to \lieg/\gad \cong [0, \infty)$$
denotes the adjoint quotient map. The map $Q$ is continuous, and is smooth outside $Q^{-1}(0).$ Set
$$ \tilde \Phi = Q \circ \Phi.$$

On the complement $\uu=\uu$ of $\Phi^{-1}(0)$ in $M,$ the map $\tilde
\Phi$ induces a Hamiltonian $S^1$-action. Explicitly, $u \in S^1 =
\rr/2\pi\zz$ acts on $m \in \uu$ by
\begin {equation}
\label {eq:action}
 m \to \exp \Bigl(u \cdot \frac{\Phi(m)}{2\tilde \Phi(m)}\Bigr) \cdot m.
\end {equation} 

This action is well-defined because $\exp(\pi H) = I$ in $\gad.$ We
can describe it alternatively as follows: on $\Phi^{-1}(\tt) \subset
M,$ it coincides with the action of $\tad \subset \gad$; then it is
extended to all of $M$ in a $\gad$-equivariant manner.

Fix $\lambda > 0.$ Using the local version (from
Remark~\ref{localcut}) of abelian symplectic cutting for the action
\eqref{eq:action}, we define the {\em non-abelian symplectic cut of
  $M$ at $\lambda$} to be
$$ M_{\leq \lambda} = \Phi^{-1}(0) \cup \uu_{\leq \lambda} = M_{< \lambda} \cup  R,$$ 
where
$$M_{< \lambda} = \Phi_1^{-1}([0, \lambda)), \ \ R_{\lambda} = \tilde
  \Phi^{-1}(\lambda)/S^1.$$ 
If $S^1$ acts freely on $\tilde \Phi^{-1}(\lambda),$ then $M_{\leq
  \lambda}$ is a smooth manifold. It can be naturally equipped with a
symplectic form $\omega_{\leq \lambda},$ coming from the symplectic
form $\omega$ on $M.$ In fact, $M_{\leq \lambda}$ is a Hamiltonian
$\gad$-manifold, cf. Remark~\ref{extraham0}. With respect to the form
$\omega_{\leq \lambda}$, $\div$ is a symplectic hypersurface in
$M_{\leq \lambda}$.

\subsection{Monotonicity} We aim to find a condition that guarantees that a non-abelian symplectic cut is monotone. As a toy model for our future results, we start with a general fact about symplectic reduction:

\begin {lemma}
\label {monquot}
Let $K$ be a Lie group with $H^2(K; \rr) = 0,$ and let $(M, \omega,
\Phi)$ be a Hamiltonian $K$-manifold that is monotone, with
monotonicity constant $\kappa.$ Assume that the moment map $\Phi$ is
proper, and the $K$-action on $\Phi^{-1}(0)$ is free. Then, the
symplectic quotient $M /\!\!/K = \Phi^{-1}(0)/K$ (with the reduced
symplectic form $\omega^{\red}$) is also monotone, with the same
monotonicity constant $\kappa.$
\end {lemma}

\begin {proof}
Consider the Kirwan map from \cite{KirwanSurjectivity}:
$$ H^2_K(M; \rr) \to H^2(M /\!\!/K; \rr),$$
which is obtained by composing the map $H^2_K(M; \rr) \to
H^2_K(\Phi^{-1}(0); \rr)$ (induced by the inclusion) with the Cartan
isomorphism $H^2_K(\Phi^{-1}(0); \rr) \cong H^2(M /\!\!/K; \rr).$ The
Kirwan map takes the first equivariant Chern class $c_1^K(TM)$ to
$c_1(T(M /\!\!/K))$, and the equivariant two-form $\tilde
\omega=\omega- \Phi$ to $\omega^{\red}.$ Since $H^2_K(M; \rr) \cong
H^2(M; \rr)$, with $c_1^K$ corresponding to $c_1$ and $[\tilde
  \omega]$ to $[\omega]$), the conclusion follows.
\end {proof}

Let us now specialize to the case when $K = \gad = PU(2).$ For
$\lambda \in (0, \infty),$ let $\ol \cong \P^1$ be the coadjoint orbit
of $\diag(i\lambda, -i\lambda),$ endowed with the
Kostant-Kirillov-Souriau form $\omega_{KKS}(\lambda)$. It has a
Hamiltonian $\gad$-action with moment map the inclusion $\iota: \ol
\to \lieg.$ Let $\gamma = \text{ P.D.}(pt)$ denote the generator of
$H^2(\ol; \zz) \subset H^2(\ol; \rr),$ so that $c_1(\ol) = 2\gamma.$
Then $c_1(\ol) = [\omega_{KKS}(1)]$, \cite[Section 7.5,Section
  7.6]{BGV}, and so $[\omega_{KKS}(\lambda)] = 2\lambda \gamma.$

 If $(M, \omega, \Phi)$ is a Hamiltonian $\gad$-manifold, let $M
 \times \ol^-$ denote the Hamiltonian manifold $(M \times \ol, \omega
 \times -\omega_{KKS}(\lambda), \Phi - \iota).$ The {\em reduction of
   $M$ with respect to $\ol$} is defined as
$$ M_{\lambda} = (M \times \ol^-) /\!\!/ \gad = \Phi^{-1}(\ol)/\gad.$$

If the $\gad$-action on $\Phi^{-1}(\ol)$ is free, the quotient
$M_{\lambda}$ is smooth and admits a natural symplectic form
$\omega_\lambda.$ It can be viewed as $\Phi^{-1}(\diag(i\lambda,
-i\lambda))/\tad;$ we let $E_{\lambda}$ denote the complex line bundle
on $M_{\lambda}$ associated to the respective $\tad$-fibration.

\begin {lemma}
\label {lemma:ol}
Let $(M, \omega, \Phi)$ be a Hamiltonian $\gad$-manifold such that the
moment map $\Phi$ is proper, and the action of $\gad$ is free outside
$\Phi^{-1}(0).$ Assume that $M$ is monotone, with monotonicity
constant $\kappa.$ Then the cohomology class of the reduced form
$\omega_{\lambda}$ is given by the formula
$$ [\omega_{\lambda}] = \kappa \cdot c_1(TM_{\lambda}) + (\lambda - \kappa) \cdot c_1(E_{\lambda}).$$  
\end {lemma}

\begin {proof}
First, note that for any Hamiltonian $\gad$-manifold $M$ we have
$H^2_{\gad}(M; \rr) \cong H^2(M; \rr),$ because $H^i(B\gad; \rr) = 0$
for $i=1,2.$ Thus the Kirwan map can viewed as going from $H^2(M;
\rr)$ into $H^2(M /\!\!/ \gad; \rr).$

Let us consider the Kirwan map for the manifold $M \times \ol^-,$
whose symplectic reduction is $M_{\lambda}.$ By abuse of notation, we
denote classes in $H^2(M)$ or $H^2(\ol^-)$ the same as their pullbacks
to $H^2(M \times \ol^-).$

Just as in the proof of Lemma~\ref{monquot}, we get that the Kirwan
map takes $[\omega] - [\omega_{KKS} (\lambda)] = \kappa c_1(TM) -
2\lambda \gamma$ to the reduced form $[\omega_{\lambda}]$, and
$c_1(TM) - c_1(T\ol)= c_1(TM) - 2\gamma$ to the reduced Chern class
$c_1(TM_{\lambda}).$ Note also that the image of $c_1(T\ol^-) =
-2\gamma$ under the Kirwan map is $c_1(E_{\lambda}).$ Hence:
$$ [\omega_{\lambda}] - \kappa \cdot c_1(TM_{\lambda}) = (\lambda -
\kappa) \cdot c_1(E_{\lambda}),$$ as desired. \end {proof}

We are now ready to study monotonicity for non-abelian cuts:

\begin {proposition} 
\label {prop:mon}
Let $\gad = PU(2),$ and $(M, \omega, \Phi)$ be a Hamiltonian
$\gad$-manifold that is monotone with monotonicity constant $\kappa >
0.$ Assume that the moment map $\Phi$ is proper, and that $\gad$ acts
freely outside $\Phi^{-1}(0).$ Then the symplectic cut $M_{\leq
  \lambda}$ at the value $\lambda = 2\kappa \in (0, \infty)$ is also
monotone, with the same monotonicity constant $\kappa.$
\end {proposition}

\begin {proof}
Recall that the symplectic cut $M_{\leq \lambda}$ is the union of the open piece $M_{< \lambda}$ and the hypersurface $R_{\lambda} = \Phi^{-1}(\ol)/S^1.$  Note that there is a natural symplectomorphism
\begin {equation}
\label {eq:symple}
\begin {CD}
 R_{\lambda} @>{\cong}>> \ol \times M_{\lambda}, \ \ \ \ m \to \bigl( \Phi(m), [m] \bigr).
 \end {CD}
\end {equation}
The inverse to this symplectomorphism is given by the map $([g], [m]) \to [gm].$

By Remark~\ref{normal}, the normal bundle to $R_{\lambda}$ is the line bundle associated to the defining $\tad$-bundle on $R_{\lambda}.$ We denote this $\tad$-bundle by $N_{\lambda}$; it is the product of $\gad \to \gad/\tad \cong \ol$ on the $\ol$ factor and the circle bundle of $E_{\lambda}$ on the $M_{\lambda}$ factor. 

Let $\nu(R_{\lambda})$ be a regular neighborhood of $R_{\lambda},$ so that the intersection $M_{< \lambda} \cap \nu(R_{\lambda})$ admits a deformation retract into a copy of $N_{\lambda}.$

 We have a Mayer-Vietoris sequence
$$ \dots \to H^1(M_{< \lambda}) \oplus H^1(\nu(\div_{\lambda})) \to
 H^1( N_{\lambda}) \to H^2(M_{\leq \lambda}) \to H^2( M_{< \lambda})
 \oplus H^2(\nu(\div_{\lambda})) \to \dots$$

Note that the first Chern class of the bundle $N_{\lambda} \to
\div_{\lambda}$ is nontorsion in $H^2(\div_{\lambda})$, because it is
so on the $\ol$ factor. Hence, the map $H^1(\nu(\div_{\lambda}); \rr)
\to H^1(N_{\lambda}; \rr)$ is onto. The Mayer-Vietoris sequence then
tells us that the map
$$ H^2(M_{\leq \lambda}; \rr) \to H^2( M_{< \lambda}; \rr) \oplus
H^2(\nu(\div_{\lambda}); \rr)$$ 
is injective. Therefore, in order to check the monotonicity of
$M_{\leq \lambda}$, it suffices to check it on $M_{< \lambda}$ and
$\nu(\div_{\lambda}).$

Since $M_{< \lambda}$ is symplectomorphic to a subset of $M,$ by
assumption monotonicity is satisfied there. Let us check it on
$\nu(\div_{\lambda})$ or, equivalently, on its deformation retract
$\div_{\lambda}.$ We will use the symplectomorphism \eqref{eq:symple}
and, by abuse of notation, we will denote the objects on $\ol$ or
$M_{\lambda}$ the same as we denote their pullback to
$\div_{\lambda}.$ Let $\gamma$ be the generator of $H^2(\ol; \zz)$ as
in the proof of Lemma~\ref{lemma:ol}. By the result of that lemma, we
have
\begin {equation}
\label {ekol}
[\omega_{\leq \lambda}|_{\div_{\lambda}}] = 2\lambda \gamma + \kappa
c_1(TM_{\lambda}) + (\lambda - \kappa) c_1(E_{\lambda}).
\end {equation}

On the other hand, the tangent space to $M_{\leq \lambda}$ at a point
of $\div_{\lambda}$ decomposes into the tangent and normal bundles to
$\div_{\lambda}.$ Therefore,
$$ c_1(TM_{\leq \lambda}|_{\div_{\lambda}}) = c_1(T\div_{\lambda}) + 2\gamma + c_1(E_{\lambda}) = 4\gamma + c_1(TM_{\lambda}) + c_1(E_{\lambda}).$$

Taking into account Equation~\eqref{ekol}, for $\lambda = 2\kappa$ we
conclude that $[\omega_{\leq \lambda}|_{\div_{\lambda}}] = \kappa
\cdot c_1(TM_{\leq \lambda}|_{\div_{\lambda}}).$
\end {proof}

\subsection {Extensions to pre-symplectic manifolds}
\label {sec:pre}

Abelian and non-abelian cutting are simply particular instances of
symplectic reduction. Since the latter can be extended to the
pre-symplectic setting, one can also define abelian and non-abelian
cutting for Hamiltonian pre-symplectic manifolds.

In general, one cannot define $c_1(TM)$ (and the notion of
monotonicity) for pre-symplectic manifolds, because there is no good
notion of compatible almost complex structure. In order to fix that,
we introduce the following:

\begin {definition}
\label {epssymp}
An {\em $\epsilon$-symplectic manifold} $(M, \{\omega_t\})$ is a
smooth manifold $M$ together with a smooth family of closed two-forms
$\omega_t \in \Omega^2(M), \ t\in [0, \epsilon]$ for some $\epsilon >
0$, such that $\omega_t$ is symplectic for all $t \in (0, \epsilon].$
\end {definition}

One should think of an $\epsilon$-symplectic manifold $(M,
\{\omega_t\})$ as the pre-symplectic manifold $(M, \omega_0)$ together
with some additional data given by the other $\omega_t$'s. In
particular, by the {\em degeneracy locus of $(M, \{\omega_t\})$} we
mean the degeneracy locus of $\omega_0$, i.e.
$$\div(\omega_0) = \{ m \in M \ | \ \omega_0 \text { is degenerate on
} T_mM \}.$$

If $(M, \{\omega_t\})$ is any $\epsilon$-symplectic manifold, we can
define its first Chern class $c_1(TM) \in H^2(M; \zz)$ by giving $TM$
an almost complex structure compatible with some $\omega_t$ for $t >
0.$ (Note that the resulting $c_1(TM)$ does not depend on $t.$) Thus,
we can define the minimal Chern number of an $\epsilon$-symplectic
manifold just as we did for symplectic manifolds. Moreover, we can
talk about monotonicity:

\begin {definition}
\label{def:es}
The $\epsilon$-symplectic manifold $(M, \{\omega_t\})$ is called {\em
  monotone} (with monotonicity constant $\kappa > 0$) if
$$ [\omega_0] = \kappa \cdot c_1(TM).$$
\end {definition}

One source of $\epsilon$-symplectic manifolds is symplectic reduction. Indeed, suppose we have a Hamiltonian pre-symplectic $S^1$-manifold $(M, \omega, \Phi)$ with the moment map $\Phi: M \to \rr$ being proper. The form $\omega$ may have some degeneracies on $\Phi^{-1}(0)$; however, we  assume that it is nondegenerate on $\Phi^{-1}\bigl((0, \epsilon]\bigr)$ for some $\epsilon > 0.$ Assume also that $S^1$ acts freely on $\Phi^{-1}\bigl([0, \epsilon]\bigr)$ (hence any $t \in (0, \epsilon]$ is a regular value for $\Phi$) and, further, $0$ is a regular value for $\Phi$ as well. Then the pre-symplectic quotients $M_t = \Phi^{-1}(t)/S^1$ for $t \in [0, \epsilon]$ form a smooth fibration over the interval $[0, \epsilon].$ By choosing a connection for this fiber bundle, we can find a smooth family of diffeomorphisms $\phi_t : M_0 \to M_t, t \in [0, \epsilon],$ with $\phi_0 = \text{id}_{M_0}.$ We can then put a structure of $\epsilon$-symplectic manifold on $M_0$ by using the forms $\phi_t^* \omega_t, t \in [0, \epsilon],$ where $\omega_t$ is the reduced form on $M_t.$ Note that the space of choices involved in this construction (i.e. connections) is contractible. Therefore, whether or not $(M_0, \phi_t^*\omega_t)$ is monotone is independent of these choices.

Since abelian and non-abelian cutting are instances of
(pre-)symplectic reduction, one can also turn pre-symplectic cuts into
$\epsilon$-symplectic manifolds in an essentially canonical way,
provided that the form is nondegenerate on the nearby cuts. (By
``nearby'' we implicitly assume that we have chosen a preferred side
for approximating the cut value: either from above or from below.) In
this context, we have the following analog of
Proposition~\ref{prop:mon}:

\begin {proposition}
\label {prop:premon}
Let $\gad = PU(2),$ and $(M, \omega, \Phi)$ be a Hamiltonian
pre-symplectic $\gad$-manifold. Set $\tilde \Phi = Q \circ \Phi : M
\to [0, \infty)$ as usual. Assume that:
\begin {itemize}
\item {The moment map $\Phi$ is proper;}
\item {The form $\omega$ is nondegenerate on the open subset $M_{< \lambda} = \tilde \Phi^{-1}\bigl([0, \lambda)\bigr),$ for some value $\lambda \in (0, \infty);$}
\item {$\gad$ acts freely on $ \tilde \Phi^{-1}\bigl((0, \lambda]\bigr)$ (hence, any $t \in (0, \lambda)$ is a regular value for $\tilde \Phi$);}
\item {$\lambda$ is also a regular value for $\tilde \Phi$;} 
\item {As a symplectic manifold, $M_{<\lambda}$ is monotone, with monotonicity constant $\kappa = \lambda/2.$}
\end {itemize}
Fix some $\epsilon \in (0, \lambda)$ and view the pre-symplectic cut
$M_{\leq \lambda}$ as an $\epsilon$-symplectic manifold, with respect
to forms $\phi_t^*\omega_{\leq \lambda - t},$ for a smooth family of
diffeomorphisms $\phi_t : M_{\leq \lambda} \to M_{\leq \lambda - t},
\ t\in [0, \epsilon], \phi_0 = id.$

Then, $M_{\leq \lambda}$ is monotone, with the same monotonicity
constant $\kappa = \lambda/2.$
\end {proposition}

\begin {proof}
We can run the same arguments as in the proof of
Proposition~\ref{prop:mon}, as long as we apply them to the
Hamiltonian manifold $M_{< \lambda},$ where $\omega$ is
nondegenerate. This gives us the corresponding formulae for the
cohomology classes $[\omega_{\leq \lambda - t}]$ and $c_1(TM_{\leq
  \lambda - t}),$ for $t \in (0, \epsilon).$ In the limit $t \to 0,$
we get monotonicity.
\end {proof}

\subsection {Cutting the extended moduli space}
\label {sec:nscut}

Recall from Section~\ref{sec:symp} that the smooth part
$\mm^\lieg_s(\Sigma')$ of the extended moduli space is a Hamiltonian
pre-symplectic $\gad$-manifold. Let us consider its non-abelian cut at
the value $\lambda = 1/2$:
$$ \nn^c(\Sigma') = \mm^\lieg_s(\Sigma')_{\leq 1/2}.$$

The notation $\nn^c(\Sigma')$ indicates that this space is a
compactification of $\nn(\Sigma') = \mm^\lieg_s(\Sigma')_{< 1/2}.$
Indeed, we have
$$ \nn^c(\Sigma') = \nn(\Sigma') \cup \div,$$
where
\begin {equation}
\label {eq:r}
\div \cong \{(A_1, B_1, \dots, A_h, B_h, \theta) \in G^{2h} \times
\lieg \ |\ \prod_{i=1}^h [A_i, B_i] = \exp(2\pi\theta) = -1 \}/S^1,
\end {equation}
Here $u \in S^1=\rr/2\pi \zz$ acts by conjugating each $A_i$ and $B_i$
by $\exp(u\theta),$ and preserving $\theta.$

The $\gad$-action on $\tilde \Phi^{-1}\bigl( (0, 1/2]\bigr) \subset
\mm^\lieg_s(\Sigma')$ is free. Since $\omega$ is nondegenerate on
$\tilde \Phi^{-1}\bigl( (0, 1/2]\bigr)$ by Theorem~\ref{jnondeg}, this
implies that any $\theta \in \lieg$ with $Q(\theta) \in (0, 1/2]$ is a
regular value for $\Phi.$ The last statement also follows from
Proposition~\ref{smoothness} (b), which further says that the values
$\theta \in \lieg$ with $Q(\theta) = 1/2$ are also regular. Hence, any
$t \in (0, 1/2]$ is a regular value for $\tilde \Phi.$ Lastly, note
that Theorem~\ref{thm:mw} says that $ \tilde \Phi^{-1}\bigl( [0,
  1/2)\bigr)$ is monotone, with monotonicity constant $\kappa = 1/4 =
  \lambda/2.$ We conclude that the hypotheses of
  Proposition~\ref{prop:premon} are satisfied.

\begin {proposition}
\label {ncmon}
Fix $\epsilon \in (0, 1/2).$ Endow $\nn^c(\Sigma')$ with the structure
of an $\epsilon$-symplectic manifold, using the forms
$\phi_t^*\omega_{\leq 1/2 -t},$ coming from a smooth family of
diffeomorphisms
$$ \phi_t : \nn^c(\Sigma') = M_{\leq 1/2} \to M_{\leq 1/2 - t}, \ t\in
[0, \epsilon], \ \phi_0 = id.$$ 
Then, $\nn^c(\Sigma')$ is monotone with monotonicity constant $1/4.$
\end {proposition}
 
Thus, we have succeeded in compactifying the symplectic manifold $\nn(\Sigma')$ while preserving monotonicity. The downside is that $\nn^c(\Sigma')$ is only pre-symplectic. The resulting two-form has degeneracies on $\div$.

\begin{lemma}
\label {prop:dcut}
Let us view $\div = \tilde \Phi^{-1}(1/2)/S^1$ as a $\P^1$-bundle over the space $\omu \times
\mm_{-I} (\Sigma'),$ compare Section~\ref{sec:degeneracies}. 
Then the null space of the form $\omega_{\leq
  1/2}$ at $x \in \div$ consists of the fiber directions.
Furthermore, the intersection number (inside $\nn^c(\Sigma)$) of $\div$ with any $\P^1$ fiber of $\div$
 is $-2$.
\end {lemma}

\begin{proof} The first claim follows from Proposition~\ref{prop:d}.
The second holds because the normal bundle is the associated bundle to
$\t^\perp$, which is a weight space with weight $-2$.
\end{proof}

In a family of forms that make $\nn^c(\Sigma')$ into an
$\epsilon$-symplectic manifold (as in Proposition~\ref{ncmon}), the
degenerate form $\omega_{\leq 1/2}$ always corresponds to $t=0.$
Hence, from now on we will denote it by $\oldo.$

\begin {proposition}
\label {ut}
In addition to the degenerate form $\oldo$ coming from the cut, the
space $\nn^c(\Sigma')=\nn(\Sigma') \cup \div$ also admits a symplectic
form $\olde$ with the following properties:

\begin {enumerate}
\item {$\div$ is a symplectic hypersurface with respect to $\olde$;}
 \item {The restrictions of $\oldo$ and $\olde$ to $\nn(\Sigma')$ have the same cohomology class in $H^2(\nn(\Sigma'); \rr);$}
 \item {The forms $\oldo$ and $\olde$ themselves coincide on the open subset $\ww = \tilde \Phi^{-1}\bigl( [0, 1/4) \bigr) \subset \nn(\Sigma');$}

\item {There exists an almost complex structure $\tilde J$ on
  $\nn^c(\Sigma')$ that preserves $\div$, is compatible with respect to
  $\olde$ on $\nn^c(\Sigma'),$ and compatible with respect to $\oldo$ on
  $\nn(\Sigma'),$ 
and for which any $\tilde J$-holomorphic sphere
of index zero has intersection number with $\div$ a negative multiple
of two.}
\end {enumerate}
\end {proposition}

\begin {proof} 
As the name suggests, the form $\olde$ will be part of a family
$(\omega_t), t\in [0, \epsilon]$ of the type used to turn
$\nn^c(\Sigma')$ into an $\epsilon$-symplectic manifold. In fact, it
is easy to find such a form that satisfies conditions (i)-(iii)
above. One needs to choose $\epsilon < 1/4$ and a smooth family of
diffeomorphisms $\phi_t: \nn^c(\Sigma') = M_{\leq 1/2} \to M_{\leq 1/2
  - t}, \ t\in [0, \epsilon], \ \phi_0 = id,$ such that $\phi_t = id$
on $\ww$ and $\phi_t$ takes $\div$ to $\div_{1/2 -t} = \tilde
\Phi^{-1}(1/2-t)/S^1;$ then set $\olde = \phi_{\epsilon}^* \oldo.$ Note
that condition (ii) is automatic from (iii), because $\ww$ is a
deformation retract of $\nn(\Sigma').$

However, in order to make sure that condition (iv) is satisfied, more care is needed in choosing the diffeomorphisms above. We will only construct $\phi=\phi_\epsilon,$ since this is all we need for our purposes; however, it will be easy to see that one could interpolate between $\phi$ and the identity. 

The strategy for constructing $\phi$ and $\tilde J$ is the same as in
the proofs of Proposition~\ref{prop:d} and Lemma~\ref{prop:dcut}: we
construct a diffeomorphism and an almost complex structure on the
toroidal extended moduli space $\mm^\tt(\Sigma')$, by looking at it as
a subset of the twisted extended moduli space
$\mm^\lieg_\tw(\Sigma');$ then, we lift them to $\mm^\lieg(\Sigma');$
finally, we show how they descend to the cut.

Let $\mu = \diag(i/2, -i/2)$ as in Section~\ref{sec:degeneracies}. We
start by carefully examining the restriction of the form $\omega$ to
$\mm^\tt(\Sigma')$, in a neighborhood of $\Phi^{-1}(\mu).$ By the
remark at the end of Section~\ref{sec:other}, this is the same as
looking at the restriction of $\omega_\tw$ to $\Phi^{-1}_\tw(\tt^*)$
in a neighborhood of $\Phi^{-1}_\tw (0).$

The zero set $Z$ of the moment map $\Phi_\tw$ on (the smooth, symplectic part of) $\mm^\lieg_\tw(\Sigma')$ is a coisotropic submanifold. Let $\omega_{\tw, 0}$ be the reduced form on $Z/\gad = \mm_{-I}(\Sigma').$ Pick a connection form $\alpha \in \Omega^1(Z) \otimes \lieg$ for the $\gad$-action on $Z.$  By the equivariant coisotropic embedding theorem \cite[Proposition 39.2]{GuilleminSternberg}, we can find a $\gad$-equivariant diffeomorphism between a neighborhood of $Z = \Phi_\tw^{-1}(0)$ in $\mm^\lieg_\tw(\Sigma')$ and a neighborhood of $Z \times \{0\}$ in $Z \times \lieg^*$ such that the form $\omega_\tw$ looks like
$$ \omega_\tw = \pi_1^* \omega_{\tw, 0} + \d(\alpha, \pi_2),$$ where
$\pi_1: Z \times \lieg \to Z \to Z/\gad$ and $\pi_2: Z \times \lieg^*
\to \lieg^*$ are projections. We can assume that $\pi_2$ corresponds
to the moment map.

Restricting this diffeomorphism to $\Phi_\tw^{-1}(\tt^*),$ we obtain a local model $Z \times \tt^*$ for that space. This implies that, locally near $Z$, we get a decomposition of its tangent spaces into several (nontrivial) bundles
\begin {equation}
\label {eq:tphi}
T(\Phi_\tw^{-1}(\tt^*)) \cong T(Z/S^1) \oplus \lieg \oplus \tt^* \cong
T( \mm_{-I}(\Sigma')) \oplus \tt^\perp \oplus (\tt \oplus \tt^*).
\end {equation}

(We omitted the pull-back symbols from notation for simplicity.)

The restriction of $\omega_\tw$ to $\Phi_\tw^{-1}(\tt^*)$ is
nondegenerate in the horizontal directions $T \mm_{-I}(\Sigma')$ as
well as on $\tt \oplus \tt^*.$ Let us compute it on the subbundle
$\tt^\perp \subset \lieg.$ For a point $x$ with $\Phi_\tw(x) = t\mu
\in \tt^*,$ and for $\xi_1, \xi_2 \in \tt^\perp \subset
T_x\Phi_\tw^{-1}(\tt^*),$ we have
\begin {equation}
\label {o12}
 \omega_\tw(\xi_1, \xi_2) = (\d \alpha(\xi_1, \xi_2), t\mu) =  -\frac{t}{2}\langle[\xi_1, \xi_2], \mu \rangle.
\end {equation}

Thus the restriction of the form to $\tt^\perp$ is nondegenerate as
long as $t\neq 0.$ (For $t=0,$ we already knew that it was degenerate
from the proof of Proposition~\ref{prop:d}.)

We construct a $\gad$-equivariant almost complex structure $J$ in a
neighborhood of $Z$ in $\Phi_\tw^{-1}(\tt^*),$ such that $J$ is split
with respect to the decomposition~\eqref{eq:tphi}, and compatible with
$\omega_\tw$ ``as much as possible.'' More precisely, we choose
$\gad$-equivariant complex structures $J_1, J_3$ on each of the
subbundles $T( \mm_{-I}(\Sigma'))$ and $\tt \oplus \tt^*$ that are
compatible with respect to the restriction of $\omega_{\tw}$ on the
respective subbundle. We also choose a $\gad$-equivariant complex
structure $J_2$ on $\tt^\perp$ that is compatible with respect to the form
$\sigma$ given by
$$\sigma(\xi_1, \xi_2) = -\langle [\xi_1, \xi_2], \mu \rangle.$$ By
Equation~\ref{o12}, we have $\omega_\tw = t\sigma/2;$ hence, $J_2$ is
compatible with respect to $\omega_\tw$ away from $t=0.$ We then let $J =
J_1 \oplus J_2 \oplus J_3$ be the almost complex structure on
$\Phi_\tw^{-1}(\tt^*)$ near $Z.$

Choose $\epsilon \in (0, 1/8)$ sufficiently small, so that $Z \times
(-3\epsilon, 3\epsilon)$ is part of the local model for
$\Phi_\tw^{-1}(\tt^*)$ described above. Pick a smooth function $f: \rr
\to \rr$ with the following properties: $f(t) = t+\epsilon$ for $t$ in
a neighborhood of $0$; $f(t) = t$ for $|t| \geq 2\epsilon;$ and $f'(t)
> 0$ everywhere. This induces a $\gad$-equivariant self-diffeomorphism
of the open subset $Z \times (-3\epsilon, 3\epsilon) \subset
\Phi_\tw^{-1}(\tt^*),$ given by $(z, t) \to (z, f(t)).$ Note that this
diffeomorphism preserves $J$, it is the identity near the boundary,
and it takes $Z \times [0, 2\epsilon)$ to $Z \times [\epsilon,
    2\epsilon).$

Now let us look at the constructions we have made in light of the
identification between $ \Phi_\tw^{-1}(\tt^*)$ and $\mm^\tt(\Sigma') =
\Phi^{-1}(\tt) \subset \mm^\lieg(\Sigma').$ We have obtained a local
model $Z \times (-3\epsilon, 3 \epsilon)$ for the neighborhood $N =
\Phi^{-1}(-3\epsilon \mu, 3 \epsilon \mu)$ of $\Phi^{-1}(\mu)$ in
$\mm^\tt(\Sigma'),$ an almost complex structure on $N$, and a
self-diffeomorphism of $N.$

The symplectic cross-section theorem \cite{GuilleminSternberg} says
that locally near $\tilde \Phi^{-1}(1/2),$ the extended moduli space
$\mm^\lieg(\Sigma')$ looks like $G \times_T \mm^\tt(\Sigma').$ Thus,
we can lift the local model for $\mm^\tt(\Sigma')$ and obtain a
$\gad$-equivariant local model $(G \times_T Z) \times (-3\epsilon, 3
\epsilon)$ for $\mm^\lieg(\Sigma').$ Projection on the second factor
corresponds to the map $1/2 - \tilde \Phi.$ Further, locally we can
decompose the tangent bundle to $\mm^\lieg(\Sigma')$ 
as in Equation~\eqref{eq:decompose}. The form $\omega$ is
nondegenerate when restricted to 
$T_{\mu}(\omu)$.  Let us choose a $\gad$-equivariant complex structure
on this subbundle that is compatible with the restriction of $\omega$
there. By combining it with $J$, we obtain an equivariant almost
complex structure $\tilde J$ on
$$ \tilde N = \tilde \Phi^{-1}(1/2 - 3\epsilon, 1/2 + 3\epsilon) \subset \mm^\lieg(\Sigma').$$

We can also lift the self-diffeomorphism of $N \subset
\mm^\tt(\Sigma')$ to $\tilde N = G \times_T N $ in an equivariant
manner. Since this self-diffeomorphism is the identity near the
boundary, we can extend it by the identity to all of
$\mm^{\lieg}_s(\Sigma').$ The result is a $\gad$-equivariant
diffeomorphism
$$ \mm^{\lieg}_s(\Sigma') \to \mm^{\lieg}_s(\Sigma') $$ 
that preserves $\tilde J$ on $\tilde N,$ takes $\tilde \Phi^{-1}(1/2)$
to $\tilde \Phi^{-1}(1/2- \epsilon),$ and is the identity on $\tilde
\Phi^{-1}\bigl( [0, 1/2-2\epsilon) \bigr).$ This diffeomorphism
  descends to one between the corresponding cut spaces:
$$ \phi: \nn^c(\Sigma') = \mm^{\lieg}_s(\Sigma')_{\leq 1/2} \to
  \mm^{\lieg}_s(\Sigma')_{\leq 1/2 - \epsilon}.$$

We set $\olde = \phi^* \oldo.$ Note that $\oldo$ and $\olde$ coincide on the
subset $ \tilde \Phi^{-1}\bigl( [0, 1/2-2\epsilon) \bigr).$ Since we
  chose $2\epsilon < 1/4,$ the latter subset contains $\ww = \tilde
  \Phi^{-1}\bigl( [0, 1/4) \bigr).$

The almost complex structure $\tilde J$ on $\tilde N$ descend to the
cut $\tilde N_{\leq 1/2}$ as well. Indeed, if $\tt \subset T\tilde N$
denotes the line bundle in the direction of the $\tad$-action used for
cutting, by construction we have $\tilde J \tt \cap T\bigl(
\tilde\Phi^{-1}(1/2)\bigr) = 0.$ Since $\tilde J$ equivariance, it is
easy to see that it induces an almost complex structure on the cut,
which we still denote $\tilde J.$ We extend $\tilde J$ to $\tilde
\Phi^{-1}\bigl( [0, 1/2-2\epsilon) \bigr)$ by choosing it to be
  compatible with $\oldo = \olde$ there. 

It is easy to see that the resulting $\tilde J$ and
  $\olde$ satisfy the required conditions (i)-(iv). With respect to the last claim in (iv), note that any $\tilde J$-holomorphic sphere of index zero is
necessarily a multiple cover of one of the fibers of the $\P^1$-bundle
$\div \to \mm_{-I} (\Sigma').$ Hence it has intersection number with
$\div$ a positive multiple of the intersection number of the fiber,
which by Lemma \ref{prop:dcut} is $-2$.
\end {proof}

\begin {remark}
\label {rem:c}
There were several choices made in the construction of $\olde$ and
$\tilde J$ in Proposition~\ref{ut}: for example, the connection
$\alpha$, the structures $J_1, J_2, J_3,$ the function $f$, etc. The
space of all these choices is contractible.
\end {remark}

\section {Symplectic instanton homology}

\subsection {Lagrangians from handlebodies} 
\label {sec:lagh}
Let $H$ be a handlebody of genus $h \geq 1$ whose boundary is the compact Riemann surface $\Sigma.$ We view $\Sigma'$ and $\Sigma$ as subsets of $H,$ with $\Sigma' = \Sigma \setminus D^2.$ 

Let $\aa^{\lieg}(\Sigma'|H) \subset \aa^{\lieg}(\Sigma')$ be the subspace of connections that extend to flat connections on the trivial $G$-bundle over $H.$ Consider also $\aa(H),$ the space of flat connections on $H$, which is acted on by the based gauge group $\gauge_0(H)=\{f: H \to G | f(z) = I \}.$ Since $\pi_1(G)=1$ and $\Sigma'$ has the homotopy type of a wedge of spheres, every map $\Sigma' \to G$ must be nullhomotopic. This implies that $\gauge^c(\Sigma')$ preserves $\aa^{\lieg}(\Sigma'|H)$ and, furthermore, the natural map
\begin {equation}
\label {flath}
 \aa(H)/\gauge_0(H) \longrightarrow  \aa^{\lieg}(\Sigma'|H)/\gauge^c(\Sigma')
\end {equation}
is a diffeomorphism.

Set
$$ L(H)=\aa(H)/\gauge_0(H) \cong \aa^{\lieg}(\Sigma'|H)/\gauge^c(\Sigma') \subset \mm^{\lieg}(\Sigma') = \aa^{\lieg}(\Sigma')/\gauge^c(\Sigma').$$

 The left hand side of \eqref{flath} is the moduli space of flat connections on $H.$ After choosing a set of $h$ simple closed curves $\alpha_1, \dots, \alpha_h$ on $H$ whose classes generate $\pi_1(H)$, the space $\aa(H)/\gauge(H)$ can be identified with the space of homomorphisms $\pi_1(H) \to G$ or, alternatively, with the Cartesian product $G^h.$

In fact, if the curves  $\alpha_1, \dots, \alpha_h$ are the same as the ones chosen on $\Sigma'$ for the identification \eqref{eq:comm}, so that the remaining curves $\beta_i$ are nullhomotopic in $H,$ then with respect to the identification \eqref{nsigma} we have
\begin {equation}
\label {lagh}
 L(H) \cong \{(A_1, B_1, \cdots, A_h, B_h) \in G^{2h} \ | \ B_i = I, \ i=1, \dots, h \} \subset \nn(\Sigma').
\end {equation}

Let us now view $L(H)$ as $\aa^{\lieg}(\Sigma'|H)/\gauge^c(\Sigma')$ via~\eqref{flath}. Note that connections $A $ that extend to $H$ in particular extend to $\Sigma$, which means that the value $\theta \in \lieg$ such that $A|_S = \theta \diff s$ is zero. In other words, $L(H)$ lies in $\Phi^{-1}(0) \subset \nn(\Sigma').$

\begin {lemma}
\label {lemma:lagr}
With respect to the Huebschmann-Jeffrey symplectic form $\omega$ from Section~\ref{sec:symp}, $L(H)$ is a Lagrangian submanifold of $\nn(\Sigma').$
\end {lemma}

\begin {proof}

Let $\tilde A$ be a flat connection on $H$ and $A$ its restriction to $\Sigma'.$
With respect to the description \eqref{tangent} of $T_{[A]}\nn(\Sigma'),$ the tangent space to $L(H)$ at $A$ consists of equivalence classes of $\diff_A$-closed forms $a \in \Omega^{1, \lieg}(\Sigma')$ which extend to $\diff_{\tilde A}$-closed forms $\tilde a \in \Omega^1(H)\otimes \lieg.$ Let $a, b$ be two such forms and $\tilde a, \tilde b$ their extensions to $H.$ We have $a|_S = b|_S = 0.$ Furthermore, by the Poincar\'e lemma for connections, on the disk $D^2$ which is the complement of $\Sigma'$ in $\Sigma$ there exists $\lambda \in \Omega^0(D^2; \lieg)$ such that $\diff_{\tilde A} \lambda = \tilde a|_{D^2}.$ By Stokes' Theorem,
$$ \int_{D^2} \langle a \wedge  b \rangle = \int_S \langle \lambda \wedge b \rangle = 0.$$

Another application of Stokes' Theorem gives
$$ \int_{\Sigma'} \langle a \wedge b \rangle = \int_{\Sigma} \langle \tilde a \wedge
\tilde b \rangle = \int_H \langle \d_{\tilde A}(\tilde a \wedge \tilde b) \rangle =0.$$

This shows that $\omega$ vanishes on the tangent space to $L(H) \cong G^{h},$ which is half-dimensional. 
\end {proof}

\subsection {Symplectic instanton homology} 
Let $Y = H_0 \cup H_1$ be a Heegaard decomposition of a three-manifold
$Y,$ where $H_0$ and $H_1$ are handlebodies of genus $h,$ with $\del
H_0 = - \del H_1 = \Sigma.$ Let $L_0 = L(H_0)$ and $L_1 = L(H_1)
\subset \nn(\Sigma')$ be the Lagrangians associated to $H_0$
resp. $H_1,$ as in Section~\ref{sec:lagh}. View $\nn(\Sigma')$ as an
open subset of the compactified space $\nn^c(\Sigma'),$ as in
Section~\ref{sec:nscut}, with $\div$ being its complement.

In Section~\ref{sec:nscut} we gave $\nn^c(\Sigma')$ the structure of
an $\epsilon$-symplectic manifold. By Lemma~\ref{prop:dcut}, its
degeneracy locus is exactly $\div.$ Using the variant of Floer
homology described in Section~\ref{sec:semi} and letting $\oo= \oldo, \oe = \olde, 
\tilde J$ be as in Proposition~\ref{ut}, we define
$$ \hsi(\Sigma'; H_0, H_1) = HF(L_0, L_1, \tilde J; R).$$ 

In order to make sure the Floer homology is well-defined, we should
check that 
the hypotheses (i)-(ix) listed at the beginning of
Section~\ref{sec:semi} are satisfied. Indeed, (i), (ii), (iii), (v)
and (x) are subsumed in Proposition~\ref{ut}. (iv), (v), and (ix)
follow from Proposition~\ref{ncmon}, Lemma~\ref{lemma:lagr} and
Corollary~\ref{cor:4}, respectively. For (viii), the Lagrangians are
simply connected and spin because they are diffeomorphic to $G^h.$ By
Theorem~\ref{thm:mw} and Lemma~\ref{lemma:z}, the minimal Chern number
of the open subset $\nn(\Sigma)$ is a multiple of $4$; therefore, the
Floer groups admit a relative $\zz/8\zz$-grading.

A priori the Floer homology depends on $\tilde J.$ However, the set of
choices used in the construction of $\tilde J$ is contractible,
cf. Remark~\ref{rem:c}. By the usual continuation arguments in Floer
theory, if we change $\tilde J$ the corresponding Floer homology
groups are canonically isomorphic.

\subsection{Dependence on the base point} \label {sec:basept}
Recall that the surface $\Sigma'$ is obtained from a closed surface
$\Sigma$ by deleting a disk around some base point $z \in \Sigma$.
Let $z_0,z_1 \in \Sigma$ be two choices of base point.  Any choice of
path $\gamma: [0,1] \to \Sigma, j \mapsto z_j, j = 0 ,1$ induces an
identification of fundamental groups $\Sigma_0' \to \Sigma_1'$, and
equivariant pre-symplectomorphisms $T_\gamma: \nn^c(\Sigma'_0) \to
\nn^c(\Sigma_1')$ preserving the cut locus $\div$. The pullbacks of
the form $\oe$ and the almost complex structure $\tilde J$ from
Proposition~\ref{ut} (applied to $\nn^c(\Sigma_1')$) can act as the
corresponding form and almost complex structure in
Proposition~\ref{ut} applied to $\nn^c(\Sigma'_0).$ Moreover, if
$H_0,H_1$ are handlebodies, the symplectomorphism $T_{\gamma}$
preserves the corresponding Lagrangians $L_0,L_1$, since the vanishing
holonomy condition is invariant under conjugation by paths.
Therefore, the continuation arguments in Floer theory show that
$T_\gamma$ induces an isomorphism
$$ \hsi(\Sigma'_0; H_0,H_1) \to \hsi(\Sigma'_1;H_0,H_1).$$

This isomorphism depends only on the homotopy class of $\gamma$
relative to its endpoints. We conclude that the symplectic instanton
homology groups naturally form a flat bundle over $\Sigma$. In
particular, there is a natural action of $\pi_1(\Sigma,z_0)$ on
$\hsi(\Sigma'_0;H_0,H_1).$

When we only care about the Floer homology group up to isomorphism
(not canonical isomorphism), we drop the base point from the notation
and write $HSI(\Sigma'; H_0, H_1) = HSI(\Sigma; H_0, H_1),$ as in the
Introduction.

\section {Invariance}

We prove here that the groups $HSI(\Sigma; H_0, H_1)$ are invariants of
the $3$-manifold $Y = H_0 \cup H_1.$ The proof is based on the theory
of Lagrangian correspondences in Floer theory,
cf. \cite{WehrheimWoodward}. We start by reviewing this theory.

\subsection{Quilted Floer homology} 

Let $M_0,M_1$ be compact symplectic manifolds. A {\em Lagrangian
  correspondence} from $M_0$ to $M_1$ is a Lagrangian submanifold
$L_{01} \subset M_0^- \times M_1$. (The minus superscript means
considering the same manifold equipped with the negative of the given
symplectic form.)  Given Lagrangian correspondences $L_{01} \subset
M_0^- \times M_1, L_{12} \subset M_1^- \times M_2$, their {\em
  composition} is the subset of $M_0^- \times M_2$ defined by
$$ L_{01} \circ L_{12} = \pi_{02}( L_{01} \times_{M_1} L_{12}) $$
where $\pi_{02}: M_0^- \times M_1 \times M_1^- \times M_2 \to M_0^-
\times M_2$ is the projection.  If the intersection
$$ L_{01} \times_{M_1} L_{12} = (L_{01} \times L_{12}) \cap (M_0^-
\times \Delta_{M_1} \times M_2) $$
is transverse (hence smooth) in $M_0^- \times M_1 \times M_1^- \times
M_2,$ and the projection $\pi_{02}: L_{01} \times_{M_1} L_{12} \to
L_{01} \circ L_{12} $ is embedded, we say that the composition $L_{02}
= L_{01} \circ L_{12}$ is {\em embedded}. An embedded composition
$L_{02}$ is a smooth Lagrangian correspondence from $M_0$ to $M_2$.

Suppose now that $M_0,M_1,M_2$ are compact symplectic manifolds,
monotone with the same monotonicity constant, and minimal Chern number
at least $2$.  Suppose that $L_0 \subset M_0, L_{01} \subset M_0^-
\times M_1, L_{12} \subset M_1^- \times M_2, L_2 \subset M_2$ are
simply connected Lagrangian submanifolds. (This implies that their
minimal Maslov numbers are at least $4.$)  
%
%C:
%
Define
$$ HF(L_0,L_{12},L_{12},L_2) := HF(L_0 \times L_{12},L_{01} \times L_2) $$
$$ HF(L_0,L_{02},L_2) := HF(L_0 \times L_2, L_{01} \circ L_{12}). $$
The main theorem of \cite{WehrheimWoodward} implies that

\begin{theorem}  \label{compose}  With $M_0,M_1,M_2,L_0,L_{01},L_{12},L_2$
monotone as above, if $L_{02} := L_{01} \circ L_{12}$ is embedded then
there exists a canonical isomorphism of Lagrangian Floer homology
groups
\begin {equation}
\label {quilts}
HF(L_0,L_{01},L_{12},L_2) \to HF(L_0,L_{02},L_2).
\end{equation} 
\end {theorem}

In Wehrheim-Woodward \cite{WehrheimWoodward} an isomorphism is defined
using pseudo-holomorphic quilts, i.e., in this case, triples of strips
in $M_0, M_1, M_2$ with boundary conditions in $L_0, L_{01}, L_{12}$
and $L_2.$ The count of such quilts is used in the left hand side of
$\eqref{quilts}.$ In the limit when the width $\delta$ of the middle
strip goes to $0,$ the same count produces the right hand side.  An
alternative proof was given in Lekili-Lipyanskiy \cite{ll:geom} using
a count of {\em $Y$-maps}.  This approach is better suited for the
semipositive case in which we will need it, so we review the
construction.  Given ${x}_- \in (L_0 \times L_{12}) \cap (L_{01}
\times L_2), {x}_+ \in (L_0 \times L_2) \cap L_{02}$ let
$\M({x}_-,{x}_+)$ denote the set of holomorphic quilts with two
strip-like ends and one cylindrical end as shown in Figure \ref{Ymap}
below, with finite energy and limits $x_\pm$.  The authors show that
for a comeagre subset of space of point-dependent compatible almost
complex structures, the moduli space $\M({x}_-,{x}_+)$ of $Y$-maps has
the structure of a finite dimensional manifold, and counting the
zero-dimensional component $\M({x}_-,{x}_+)_0$ defines a cochain map
$$ \Phi: CF(L_0,L_{01},L_{12},L_2) \to CF(L_0,L_{02},L_2), \quad
\bra{{x}_-} \mapsto \sum_{u \in \M({x}_-,{x}_+)_0} \eps(u) \bra{{x}_+}
.$$
Here, in the case of integer coefficients, the map
$$ \eps: \M({x}_-,{x}_+)_0  \to \{ \pm 1 \} $$
is defined by comparing the orientations constructed in \cite{orient}
with the canonical orientation of a point. 

\begin{figure}
\begin{center}
\input{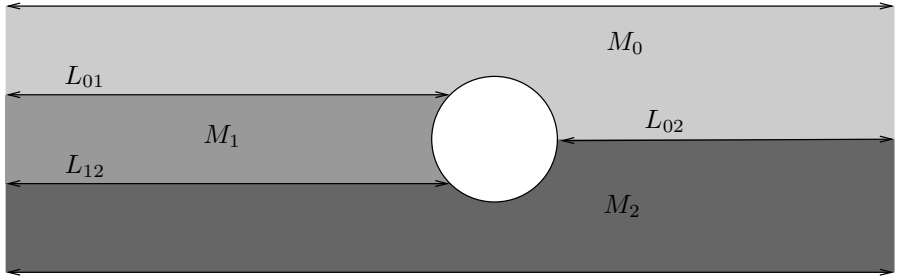}
\end{center}
\caption{Geometric composition via a quilt count of $Y$-maps}
\label{quiltcount} 
\label{Ymap}
\end{figure} 

\noindent Counting $Y$-maps in the opposite direction defines a chain
map $$\Psi: CF(L_0,L_{02},L_2) \to CF(L_0,L_{01},L_{12},L_2) .$$
Lekili-Lipyanskiy \cite{ll:geom} prove that the monotonicity constant
for these $Y$-maps is the same as the monotonicity constant for Floer
trajectories.  They then show that $\Phi$ and $\Psi$ induce
isomorphisms on homology.

\subsection{Relative quilted Floer homology in semipositive manifolds} 

We wish to have a version of the quilted Floer homology and Composition
Theorem \ref{compose} which holds for Floer homology {\em relative to
  hypersurfaces} in {\em semipositive} manifolds, as in Section~\ref{sec:semi}.  Suppose
that $R_0,R_1$ are symplectic hypersurfaces in $M_0,M_1$. From them we
obtain two hypersurfaces $\tilde R_0 = R_0^- \times M_1$, $\tilde R_1
= M_0^- \times R_1$ in $M_0^- \times M_1.$ Let $\ti{N}_{R_0},
\ti{N}_{R_1}$ denote their normal bundles $N_{R_0},N_{R_1}$, that is,
the pull-backs of $N_{R_0},N_{R_1}$ to $\ti{R}_0,\ti{R}_1$.  Because
$R_0,R_1$ are symplectic, $N_{R_0}, N_{R_1}$ are {\em oriented} rank
$2$ bundles, or equivalently up to homotopy, rank one complex line
bundles.  As we will see below, the following definition gives
sufficient conditions for a sort of combined intersection number with
$R_0,R_1$ to be well-defined and given by the usual geometric
formulas:

\begin{definition} \label{compatcorr} 
A simply connected Lagrangian correspondence $L_{01} \subset M_0^-
\times M_1 $ is called {\em compatible} with the pair $(R_0,R_1)$ if 
$$ (R_0 \times M_1) \cap L_{01} = (M_0 \times R_1) \cap L_{01} = (R_0
\times R_1) \cap L_{01}$$ 
and there exists an isomorphism
$$\varphi: \bigl ( \ti{N}_{R_0} \bigr)  |_{ (R_0 \times R_1) \cap L_{01}} \cong
\bigl( \ti{N}_{R_1} \bigr) |_{ (R_0 \times R_1) \cap L_{01} }$$
and tubular neighborhoods
$$\tau_0: N_{R_0} \to M_0, \quad \tau_1: N_{R_1} \to M_1$$ 
of $R_0$ resp. $R_1$ such that $(\tau_0 \times \tau_1)^{-1}(L_{01})
\subset N_{R_0} \times N_{R_1} = \ti{N}_{R_0} \times_{M_0 \times M_1}
\ti{N}_{R_1}$ is equal to the graph of $\varphi$.
\end{definition} 

To explain the conditions in the definition, note that the existence
of $\varphi$ implies that any map of a compact oriented surface with
boundary to $M$ with boundary conditions in $L_{01}$ has a
well-defined intersection number with $\ti{R}_0 \cup \ti{R}_1$.  For example, suppose that $u:
(D,\partial D) \to (M_0 \times M_1,L_{01})$ is a disk with Lagrangian
boundary conditions. The sum of dual classes $[\tilde R_0]^\dual
+[\tilde R_1]^\dual$ has trivial restriction to $H^2(L_{01}; \zz)$. If
$L_{01}$ is simply connected then $H^2(M,L_{01})$ is the kernel of
$H^2(M) \to H^2(L_{01})$ and so we may consider $[\tilde R_0]^\dual
+[\tilde R_1]^\dual$ as a class in $H^2(M,L_{01})$.  Then the
intersection number of a map $u: (D,\partial D) \to (M_0 \times
M_1,L_{01})$ with $[\tilde R_0] +[\tilde R_1]$ is
well-defined and denoted $u\cdot R$.

The existence of the tubular neighborhoods $\tau_0,\tau_1$ implies
that $u\cdot R$ is given by a geometric count of intersection points.  Indeed, we
may identify a neighborhood of $R_j$ with the normal bundle $\pi_j:
N_j \to R_j$ via the tubular neighborhood $\tau_j$.  Then $\pi_j^*N_j$
is trivial on the complement of $R_j$, since the map $\pi_j$ gives a
non-vanishing section, and extends to a bundle $L_{R_j}$ on $M_j$
trivial on the complement of $R_j$.  Then the dual class $[R_j]^\dual$
is given by a Thom class in the tubular neighborhood of $R_j$, and
hence equals the Euler class of $L_{R_j}$.  The bundles $ L_{R_0}$ and
$L_{R_1}$ are isomorphic on $\partial D$ via $\varphi$, and so glue
together to a bundle denoted $u^* L_R$ over $S^2 = D \cup_{\partial D}
D$.  The intersection number is then the Euler number of $u^* L_R$,
that is, 
$$u\cdot R = ([S^2], \on{Eul}(u^* L_R)).$$  
The compatibility
condition on the maps $\tau_j$ implies that the maps $u_0,u_1$
considered as sections of $L_{R_j}$ near $R_j$ glue together to a
section of $u^* L_R$, which by abuse of notation we denote by $u$.  If
each $u_j$ meets $R_j$ in a finite number of points then $u \cdot R$ is a
sum of local intersection numbers $(u \cdot R)_z$, given by the image of a
small loop around each intersection point $z$ in $H_1( u^* L_R |_V - 0
,\Z) \cong \Z$ in a small neighborhood $V$ of $z$.  Note that since we
have constructed $u^* L_R$ only as a {\em topological} (or rather,
piecewise smooth) bundle, such a loop will only be piecewise smooth if
$z \in \partial D$.

Our examples will arise as follows:

\begin {example}
\label{ex:coiso}
Suppose $\iota: C \to M_1$ is a fibered coisotropic submanifold of
$M_1$, with structure group $C$ the fibration being $\pi: C \to M_0.$
Then $(\pi \times \iota):C \to M_0^- \times M_1$ defines a Lagrangian
correspondence, compare \cite[Example 2.0.3(b)]{WehrheimWoodward}.
Suppose further that $M_1$ is a Hamiltonian $U(1)$-manifold with
moment map $\Phi_1$ and $C$ is $U(1)$-invariant and meets
$\Phi_0^{-1}(\lambda)$ transversely. Then the symplectic cut
$M_{1,\leq \lambda}$ contains the closure $C_{\leq \lambda}$ of the
image $C$ as a fibered coisotropic, whose graph is a Lagrangian
correspondence in $M_{0,\leq \lambda} \times M_{1,\leq \lambda}$.
Furthermore, the submanifolds $R_0 := M_{0,\lambda}, R_1: =
M_{1,\lambda}$ are symplectic submanifolds with the properties
described in Definition \ref{compatcorr}.  Indeed, any tubular
neighborhood $N_{R_1} \to M_{1,\leq \lambda}$ of $R_1$ that is
$U(1)$-invariant, maps $N_{R_1} |_{C_{\leq \lambda}}$ to $C_{\leq
  \lambda}$, and maps fibers to fibers induces a tubular neighborhood
$N_{R_0} \to M_{0,\leq \lambda}$ with the required properties.
 \end{example}

The intersection numbers described above are well-defined more
generally for quilted strips, as we now explain.  Given symplectic
manifolds $M_1,\ldots,M_k$, Lagrangian submanifolds $L_1 \subset M_1,L_k \subset M_k$ disjoint from $R_1$ resp. $R_k$, and Lagrangian  correspondences
$$ L_{12} \subset M_1 \times M_2, \ldots, L_{(k-1)k} \subset M_{k-1}^-
\times M_k $$
compatible with hypersurfaces
$\ul{R} = (R_j \subset M_j)_{ j = 1,\ldots, k} $
the intersection number $\ul{u} \cdot \ul{R}$ of a quilted Floer
trajectory
$$\ul{u} = (u_j: \R \times [0,1] \to M_j )_{j=1}^k $$
is the pairing of $\ul{u}$ with the sum of the dual classes
$[R_j]^\dual$ to $R_j$.  

If the intersection of $\ul{u}$ with $\ul{R}$ is finite, then the
intersection number is the sum of local intersection numbers defined
as follows.  By assumption, there exists an isomorphism 
$$N_{j-1} |_{
L_{(j-1)j} \cap (R_{j-1} \times R_j)} \xrightarrow{\cong} N_j |_{ L_{(j-1)j} \cap
(R_{j-1} \times R_j)}$$ 
and this extends to an isomorphism of
$\ti{N}_{R_j} |_{ L_{(j-1)j}}$ and $\ti{N}_{R_{j-1}} |_{ L_{(j-1)j}}$, by
the assumption about the tubular neighborhoods.  Thus the pull-back
bundles $ u_j^* \ti{N}_{R_j} $ patch together to a bundle on the
quilted surface $\ul{S} = \cup_j S_j$ that we denote by $ \ul{u}^*
\ti{N}_{\ul{R}}$.  The intersection number is then the relative Euler
number of $\ul{u}^* \ul{N}_{\ul{R}} \to \ul{S}$, that is, the pairing
of the relative Euler class with the generator of $H^2(cl(\ul{S}),\partial cl(\ul{S}))$ where 
$cl(\ul{S})$ is the closed disk
obtained by adding points at $\pm \infty$.  The map $\ul{u}$ then
provides a section of $\ul{u}^* \ul{N}_{\ul{R}}$, by the compatibility
conditions in Definition \ref{compatcorr}.  If the intersection is
finite, then
\begin{equation} \label{local} \ul{u}\cdot \ul{R} = \sum_{\{z \in \ul{S}| \ul{u}(z) \in \ul{R}\} }
(\ul{u} \cdot \ul{R})_z \end{equation}
where $(\ul{u}\cdot \ul{R})_z \in \Z$ is, as in the case of disks discussed
before, the image of a small loop around $z$ in the complement
$\ul{u}^* \ul{N}_{\ul{R}} - 0$ of the zero section, as a multiple of
the generator of the first homology of the fiber, and the condition
$\ul{u}(z) \in \ul{R}$ means that if $z$ lies in the component $S_j$,
then $u_j(z_j) \in R_j$.  Note that in particular that these local
intersection numbers are topologically continuous, that is, given any
loop in the domain of the quilt the sum of the local intersection numbers is
constant in any continuous family as long as none of the intersection points
cross the loop.

If the intersection is not only finite but transverse, and the
hypersurfaces $\ul{R}$ are almost complex, then the intersection
number is the usual one counted with weight $1/2$ for the seam points:

\begin{lemma} \label{formula} 
 Suppose that $L_0$ resp. $L_k$ is disjoint from $R_0$ resp. $R_k$ and
 each $L_{(j-1)j}$ is compatible with $(R_{j-1},R_j)$.  Suppose that
 the almost complex structure on $M_0 \times \ldots \times M_k$ is of
 product form $J_0 \times \ldots \times J_k$ near each $\ti{R}_j$, so
 that each $R_j$ is an almost complex submanifold of $M_j$ with
 respect to $J_j$.  Let $\ul{u}: \ul{S} \to \ul{M}$ be a quilted Floer
 trajectory with Lagrangian boundary and seam conditions in $\ul{L}$
 meeting each $\ti{R}_j$ transversally.  Then
$$ \ul{u} \cdot \ul{R} = 
\sum_{j=0}^k 
\# \{ z_j \in \on{int}(S_j) | u_j(z_j)
\in R_j \} + \frac{1}{2} \# \{ z_j \in \partial S_j | u_j(z_j) \in R_j \} .$$
\end{lemma} 

\begin{proof}  
The local intersection number in \eqref{local} at a transversal point
of intersection $z \in \ul{S}$ is the homology class of the image of a
small loop around $z$, considered as an element of $H_1(\ul{N}_z)
\cong \Z$.  We consider only the case of an intersection point $z$ on
the seam, the loop is divided into two loops, one coming from each
component of the quilt, and is only piecewise smooth; the case of an
interior intersection is easier and left to the reader.  

Suppose $z$ is on the seam $L_{(j-1)j}$ where the components $u_{j-1}$ and $u_j$ of the quilt meet. For $l=j-1$ or $j,$ let us view $u_l$ as a section of a piecewise smooth line bundle. Using a local trivialization of the bundle and a coordinate chart for $\ul{S}$ centered at $z$, we have
that $u_l$ near $z$ (now viewed as a map to $\cc$) is given approximately by its linearization at $z$: 
$$|u_l(r\exp(it)) - (D u_l(z)) r \exp(it)| < Cr^2. $$ 

We use here that since $R_l$ is almost complex, the linearization
$Du_l$ is complex linear.  Fix $\epsilon > 0.$ For $r$ sufficiently
small, we have
$$|\arg (u_l(r \exp (it))) - \arg( Du_l(z) r \exp(it))| < \epsilon.$$
This implies that
$$ |\int_{0}^1 u_l^* \d \theta  -  \pi | <  \epsilon, \quad l = j-1,j $$ 
and so 
$$ \int_{0}^1 u_{j-1}^* \d \theta  +
\int_{0}^1 u_j^* \d \theta   \in (2\pi-2\epsilon, 2 \pi+ 2\epsilon) .$$
Since the integral must be an integer multiple of $2\pi$ (and
$\epsilon$ can be chosen arbitrarily small), the integral must in fact
equal $2\pi$.  It follows that the two paths patch together to a
positive generator of $H_1(\C^*,\Z)$, as claimed.
\end{proof} 

We can now define relative quilted Floer homology in semipositive
manifolds.

\begin{theorem} \label{corrdef}
 Suppose that $\ul{M} = (M_i)_{i=0}^k$ are semipositive manifolds as
 in the first six items of Assumptions \ref{assumptions}, with a
 collection of open sets $\ul{\opens} = (\opens_i)_{i=0}^k$ on which
 the respective forms $\omega_i$ and $\ti{\omega}_i$ coincide. Suppose
 the manifolds $M_i$ come equipped with almost complex structures
 $\ti{J}_i$, so that the degeneracy loci $R_i$ of the forms
 $\ti{\omega}_i$ are almost complex hypersurfaces in $M_i$, 
 disjoint from $\opens_i$.  We denote
 by $\jt(\ul{M}, \ul{\opens}, \ti{J})$ the space of time-dependent
 almost complex structures on $M_0 \times \ldots \times M_k$ that
 agree with $\ti{J} = \ti{J}_0 \times \dots \times \ti{J}_k$ on
 $\opens := \prod_{i=0}^k \opens_i.$

 We are also given simply connected Lagrangians $L_0
 \subset M_0, L_{01} \subset M_0^- \times M_1, \ldots, L_{(k-1)k}
 \subset M_{k-1}^- \times M_k, L_k \subset M_k$ such that the seam
 conditions $L_{(i-1)i}$ are compatible with $(R_{i-1}, R_i)$ and
 $L_0$ resp. $L_k$ 
  are contained in $\opens_0$ resp. $\opens_k.$ Also, we assume that
 \begin {equation}
 \label {eq:L_opens}
  (L_0 \times L_{12} \times \dots) \cap (L_{01} \times L_{23} \times \dots) \subset \opens_0 \times \dots \times \opens_k.
\end {equation}

 Suppose
 further that any holomorphic disk with boundary in $L_{(i-1)i}, i
 = 1,\ldots,k$ or holomorphic sphere with zero canonical area has
 intersection number with $\ul{R}$ given by a negative multiple of
 $2$. 
 
 Then, there exists a comeagre subset
 $\jt^{\reg}(\ul{L},\ul{\opens},\ti{J})$  of  $\jt(\ul{M}, \ul{\opens}, \ti{J})$ so that if the almost complex structure 
$(J_t)$ is chosen from
$\jt^{\reg}(\ul{L},\ul{\opens},\ti{J})$ then the part of the Floer
differential of $CF(\ul{L}) = CF(L_0 \times L_{12} \times \ldots ,
L_{01} \times L_{23} \times \ldots)$ counting trajectories disjoint
from $R_i,i = 1,\ldots,m$, is finite and squares to zero.  We denote
by $$HF(\ul{L} ;\ul{R}) := HF(\ul{L}, \ti{J};\ul{R})$$ the resulting
 Floer homology group; it is independent up to isomorphism on all
 choices except possibly the base almost complex structures
 $\ti{J}_i$.
\end{theorem} 

\begin{proof}   
First, note that the condition \eqref{eq:L_opens} implies that the endpoints of any holomorphic quilt are contained in $\opens = \opens_0 \times \dots \times \opens_k.$ Hence, every quilt component $u_i$ contains a point in the respective 
open set $\opens_i.$ This implies that the usual transversality arguments for holomorphic
quilts apply, even when we restrict to almost complex structures $J_t$
that are required to agree with $\ti{J}$ on $\opens.$

Next, we discuss compactness. We must rule out sphere and disk bubbling
in the zero and one-dimensional moduli spaces.  For a suitable
comeagre subset of almost complex structures agreeing with the given
$\ti{J}_i$, the trajectories are transverse to the $R_j$ in the zero
and one-dimensional moduli spaces, by the same argument we gave
previously for the unquilted case (Corollary~\ref{cor:transv}). 

 Suppose that $\ul{u}_\infty$ is
the limit of a sequence of trajectories of index $1$ or $2$ disjoint
from $\ul{R}$.  By the assumption on the intersection number, any
sphere bubble or disk bubble with boundary in some $L_{(j-1)j}$
contributes at least $-2$ to the intersection number with $\ul{R}$.
It follows that at least one intersection point does not have a bubble
attached.  But then, since the intersection point is transverse,
$\ul{u}_\infty$ cannot be the limit of a sequence of trajectories
disjoint from $\ul{R}$, since transverse intersection points persist
under deformation.  Hence there is no such bubbling and the limit is a
(possibly broken) trajectory, as desired.  Independence of the choice
of almost complex structures is proved by the usual continuation
argument, ruling out disk bubbles of index one and sphere bubbles by
the same reasoning.
\end{proof}  

\begin{remark}   If the Lagrangian correspondences above are associated
to fibered coisotropics, then the almost complex structures may be
taken of split form, that is, products of the almost complex
structures on $M_0,\ldots,M_k$.  This will be the case in our
application.
\end{remark} 

\begin{theorem} 
 \label{relcompose}
 Suppose that $\ul{M} = (M_0, M_1, M_2)$ and $\ul{L} = (L_0,L_{01},L_{12},L_2)$ satisfy the assumptions in Theorem \ref{corrdef}. Suppose further that $L_{01} \circ L_{12}$
 is an embedded composition, is simply-connected, and is compatible
 with $(R_0, R_2),$ and that all holomorphic quilted cylinders with
 seams in $L_{01},L_{12}, L_{01} \circ L_{12}$ with zero canonical
 area have intersection number equal to a negative multiple of $2$.
Then the relative Lagrangian Floer homology groups
$HF(L_0,L_{01},L_{12},L_2;R_0,R_1,R_2)$, $HF(L_0,L_{01} \circ
L_{12},L_2;R_0,R_2)$ are isomorphic.  Similar statements hold for the
composition of any two adjacent pairs, as long as the compositions are
smooth and embedded.
\end{theorem} 

\begin{proof} 
 If the Lagrangian correspondences had been monotone, the result would
 have been a slight extension of Theorem~\ref{compose} in
 \cite{WehrheimWoodward}, by counting only those trajectories disjoint
 from $R_i$; indeed, since the intersection numbers are homotopy
 invariants, they do not change when taking the limit $\delta \to 0.$
 
 In the semipositive case at hand, one can rule out disk and sphere
 bubbling as in the proof of Proposition~\ref{prop:semifloer}, but not
 the figure eight bubbles mentioned in \cite[Section
   5.3]{WehrheimWoodward}. Indeed, removal of singularities,
 transversality, and Fredholm theory for figure eight bubbles have not
 yet been developed.  For this reason, we use instead the approach of
 Lekili-Lipyanskiy \cite{ll:geom}.

First one checks that for a comeagre subset of compatible almost
complex structures, the ends of the cylinders of $Y$-maps will not map
to $R$ in the $0$ and $1$-dimensional components of the moduli space,
since this is a codimension $2$ condition.  Indeed, an examination of
the weighted Sobolev space construction of the moduli space of
$Y$-maps in \cite{ll:geom} shows that the evaluation map at the end of
the cylinder is smooth; indeed it projects onto the factor of
asymptotically constant maps in the Banach manifolds in which the
moduli space of $Y$-maps is locally embedded: $W^{1,p,\eps}( S;
\ul{u}^* T\ul{M}, \ul{u}^* T\ul{L}) \oplus T_{ (\ul{u})_{02}(\infty)}
L_{02}$, where the former is the space of from $S$ with Lagrangian
boundary conditions with finite $\eps$-weighted Sobolev norm of class
$(1,p)$ and the latter is the intersection of the linearized
Lagrangian boundary conditions at infinity on the cylindrical end.

As a result, the intersection number $\ul{u}\cdot \ul{R}$ of any $Y$-map
$\ul{u}$ of index zero and one with the collection $\ul{R}$ is
well-defined and given by the formula \eqref{formula}.  (More
generally, one could make the intersection number with {\em any}
$Y$-map well-defined by imposing the compatibility condition
$\varphi_{01} \circ \varphi_{12} = \varphi_{02}$, so that the bundle
$\ul{u}^* L_{\ul{R}}$ is well-defined. But we will not need this.)
In the zero and one dimensional moduli spaces all intersections with
the manifolds $R_j$ are transverse for $\ul{J}$ chosen from a comeagre
subset of the space of compatible almost complex structures making
$R_j$ almost complex, by standard arguments \cite[Section
  6]{cm:tr}.  
  
A Gromov compactness argument shows that finite energy Y-maps have as
limits configurations consisting of a (possibly broken) Y-map together
with some sphere bubbles, disk bubbles, and cylinder bubbles. The
cylinder bubbles may form when there is an accumulation of energy at
the Y-end.
  
In the case at hand, sphere and disk bubbles are ruled out as as in the proof of 
Theorem ~\ref{corrdef}: any sphere or disk bubble appearing in the
limit configuration $\ul{u}_\infty$ must have index zero, and
therefore intersection number at most $-2$ with $\ul{R}$.  By
\eqref{formula}, any intersection point contributes at most $1$ to the
intersection number, and therefore at some intersection point with
$\ul{R}$ is not attached to a bubble.  But then $\ul{u}_\infty$ cannot
be the limit of a sequence of trajectories disjoint from $\ul{R}$,
since the local intersection number of $\ul{u}_\infty$ is non-zero.

It remains to rule out cylinder bubbles. 
Since no trajectory of index zero or one maps the end of the cylinder
to $\ul{R}$, any quilted cylinder bubble must capture positive canonical
area.  But then, for index reasons explained in Lekili-Lipyanskiy
\cite{ll:geom}, the cylinder bubble must capture at least index two,
so the index of the remaining $Y$-map is at most $-1$.  (Here working
with $Y$-maps, rather than strip-shrinking, provides an advantage: by
exponential decay for holomorphic strips with boundary values in
Lagrangians intersecting cleanly, one knows that these cylinder
bubbles connect to a point outside of $\ul{R}$, whereas for figure
eight bubbles such exponential decay estimates are missing.)  But such
a trajectory does not exist, since transversality is achieved for the
chosen $\ul{J}$.

It follows that the moduli spaces of $Y$-maps of dimension zero and
one that are disjoint from $\ul{R}$ are compact up to breaking off
trajectories disjoint from $\ul{R}$.  Furthermore, for these
trajectories and $Y$-maps we have the same relationship as in
\cite{ll:geom}, since the complements of $\ul{R}$ are monotone. The
rest of their argument now goes as in \cite[Section 3.1]{ll:geom}.
\end{proof}

\subsection{Proof of invariance} 

Going back to topology, let $\Sigma_0, \Sigma_1$ be Riemann surfaces
of genus $h$, resp. $h + 1$. Let $H_{01}$ be a compression body with
boundary $\Sigma_0^- \times \Sigma_1$, that is, a cobordism consisting
of attaching a single handle of index one.  Associated to $H_{01}$ we
have a Lagrangian correspondence
$$ L_{01} \subset \nn(\Sigma_0')^- \times \nn(\Sigma_1') $$
defined as follows.  Suppose that $\gamma$ is a path from the base
points $z_0$ to $z_1$, equipped with a framing of the normal bundle.
Let $H_{01}'$ denote the non-compact surface obtained from $H_{01}$ by
removing a regular neighborhood of $\gamma$. The boundary of $H_{01}'$ then consists 
of $\Sigma_0', \Sigma_1'$ and a cylinder $S \times [0,1].$ Let $\nn(H_{01}')$ denote the moduli space of flat
connections on $H_{01}$ of the form $\theta \d s$ near $S \times [0,1]$ (where $s$ is the coordinate on the circle $S$), for
some $\theta \in \g$, modulo gauge transformations equal to the identity
in a neighborhood of $S \times [0,1].$  The same arguments as in the
proof of Lemma~\ref{lemma:lagr} show that $L_{01}$ is a Lagrangian correspondence.

The Lagrangian correspondence $L_{01}$ has the following explicit
description in terms of holonomies, similar to \eqref{nsigma} and ~\eqref{lagh}.  Suppose that $H_{01}$ consists of attaching a one-handle
whose meridian is the generator $B_{h+1}$ of $\pi_1(\Sigma_1)$. 
Then:

\begin{lemma} \label{holcor} 
The Lagrangian correspondence $L_{01}$ is given by
$$ L_{01} = \{ ((A_1,\ldots,B_h) \in \nn(\Sigma_0'),
(A_1,\ldots,B_h, A_{h+1}, B_{h+1}) \in \nn(\Sigma_1'))\ | \ B_{h+1} = I \} .$$
\end{lemma} 

\begin{proof} $H_{01}'$ has the homotopy type of the wedge product of 
$\Sigma_0'$ with a circle, corresponding to a single additional
generator $a_{h+1}$. Thus $\pi_1(H_{01}')$ is freely generated by $(a_1,\ldots,b_h,a_{h+1})$, and the lemma follows.
\end{proof}  

Recall from Section~\ref{sec:nscut} that $\nn(\Sigma_0')$ admits a compactification $\nn^c(\Sigma_0') = \nn(\Sigma_0') \cup R_0.$ We equip $\nn^c(\Sigma')$ with the (non-monotone) symplectic form constructed in Proposition~\ref{ut}, which we denote by $\omega_{\epsilon, 0}.$ Then $R_0$ is a symplectic hypersurface. Similarly, we have a symplectic form $\omega_{\epsilon, 1}$ on $\nn^c(\Sigma'_1) = \nn(\Sigma'_1) \cup R_1.$ Let $L_{01}^\c$ denote the closure of $L_{01}$ in the compactification $\nn^\c(\Sigma_0')^- \times \nn^\c(\Sigma_1')$.

\begin{lemma} \label{compat} The Lagrangian correspondence $L_{01}^c$ is compatible with the pair $(R_0,R_1)$.  
Furthermore, any disk bubble with boundary 
in $L_{01}^c$ with index zero has intersection number with $(R_0,R_1)$
a negative multiple of $2$.
\end{lemma} 

\begin{proof} 
View $L_{01}^c$ 
as a coisotropic submanifold of $\nn^\c(\Sigma_1')$,
fibered over $\nn^c(\Sigma_0')$ with fiber $G.$ We are then exactly in
the setting of Example~\ref{ex:coiso}.  
To prove the claim on the intersection number, note that any fiber of
$R_1$ which intersects $L_{01}$ is mapped symplectomorphically onto
the corresponding fiber of $R_0$ via the projection of the fibered
coisotropic $B_{h+1} = I$.  Hence the patches of any such 
disk bubble, after projection to $\nn^c(\Sigma_0')$, glue together to a
sphere bubble in the $\P^1$-fiber of $R_0$.  Furthermore, the
projection induces an isomorphism of normal bundles by assumption, so
the intersection number is equal to the intersection number of the
sphere with $R_0$, which is a negative multiple of $2$ as claimed.
\end{proof}

\begin{lemma}  \label{comp1} Let $L_0 \subset \nn^c(\Sigma_0')$, 
resp.  $L_1 \subset \nn^c(\Sigma_1'),$ be the Lagrangian for the
handlebody given by contracting the cycles $b_1,\ldots,b_h$,
resp. $b_1,\ldots,b_{h+1}$.  Then the composition $L_0 \circ L_{01}^\c$ is embedded, and equals $L_1$.
\end{lemma}

\begin{proof} Immediate from Lemma \ref{holcor} and the fact
that $L_0$ does not meet the hypersurface $R_0$.
\end{proof}  

\begin{lemma} \label{comp2} 
Let $L_{01}^c \subset \nn^c(\Sigma_0')^- \times \nn^c(\Sigma_1')$ be
the Lagrangian correspondence for attaching a handle corresponding to
adding the cycle $a_{h+1}$, and $L^c_{10} \subset \nn^c(\Sigma_1')^-
\times \nn^c(\Sigma_0')$ the Lagrangian correspondence corresponding
to contracting the cycle $b_{h+1}$.  Then the composition $L_{01}^\c
\circ L_{10}^\c$ is embedded, and equals the diagonal $\Delta_0
\subset \nn^c(\Sigma_0')^- \times \nn^c(\Sigma_0')$. 
Furthermore, any quilted cylinder with seams in $L_{10}^\c,L_{01}^\c,
\Delta_0$ with index zero has intersection number with $(R_0,R_1,R_0)$
a negative multiple of $2$.
\end{lemma} 

\begin{proof} The first claim is immediate from Lemma \ref{holcor}.
To see the assertion on the quilted cylinders, note that any quilted
cylinder of index zero has zero canonical area, and so each component
is contained in the corresponding $R_j$ and maps onto a single fiber
of the degeneracy locus.  As in the proof of Lemma \ref{compat}, the three holomorphic strips
patch together to an orientation-preserving map of a sphere to a fiber
of $R_0$, which must have intersection number a positive multiple of
the intersection number of the fiber, which is $-2$.
\end{proof}  

\begin {proof}[Proof of Theorem~\ref{thm:Invariance}]
We seek to show that the Floer homology groups
$$HSI(\Sigma'; H_0, H_1) = HF(L_0,L_1; R)$$ are independent of the
choice of Heegaard splitting of the $3$-manifold $Y$.

By the Reidemeister-Singer theorem (\cite{Reidemeister},
\cite{Singer}), any two Heegaard splittings $Y = H_0 \cup_{\Sigma_0}
H_1$, $Y = H'_0 \cup_{\Sigma_1} H'_1$, are related by a sequence of
stabilizations and de-stabilizations.  Therefore it suffices to
consider the case that $H'_0,H'_1$ are obtained from $H_0,H_1$ by
stabilization.  That is,
$$ {H}'_0 = H_0 \cup_{\Sigma_0} H_{01}, 
 \ \ \ {H}'_1 = H_1 \cup_{\Sigma_0} (-H_{10}) $$
where $H_{01}, H_{10}$ are the compression bodies corresponding to
adding the cycle $a_{h+1}$, resp. contracting $b_{h+1}$.  Then, after
three applications of Theorem \ref{relcompose}, and taking into
account Lemmas \ref{comp1}, \ref{comp2}, we have
\begin{eqnarray*} 
HF(L_0,L_1; R_0) &\cong& HF(L_0 ,\Delta_0, L_1;R_0,R_0) \\
            &\cong& HF(L_0 ,L^c_{01} , L_{10}^c, L_1;R_0,R_1,R_0) \\
            &\cong& HF(L_0 \circ L^c_{01}, L^c_{10} \circ L_1; R_1,R_1) \\
            &=& HF(L'_0,L'_1;R_1). 
\end{eqnarray*} \end {proof}

\begin {remark}
The symplectic instanton homology groups $\hsi(Y, z)$ depend on the
choice of basepoint $z \in \Sigma \subset Y,$ compare
Section~\ref{sec:basept}. As $z$ varies, the groups naturally form a
flat bundle over $Y.$ Still, we usually drop $z$ from the notation and
denote them as $\hsi(Y).$
\end {remark}

\section {Properties and examples}
\subsection {The Euler characterstic}
In general, the Euler characteristic of Lagrangian Floer homology is
the intersection number of the two Lagrangians. In our situation, the
corresponding intersection number is computed (up to a sign) in
\cite[Proposition 1.1 (a), (b)]{AkbulutM}:
\begin {equation}
\label {eq:euler}
\chi \bigl( HSI(Y) \bigr) = [L_0] \cdot [L_1] =  \begin {cases} 
\pm |H_1(Y; \zz)| & \text{ if } b_1(Y)=0; \\
0 & \text{ otherwise}.
\end {cases}
\end {equation}

\subsection {Examples}
\label {sec:examples}

\begin {proposition}
We have an isomorphism
$$ \hsi(S^3) \cong \zz.$$
\end {proposition}

\begin {proof}
Let $\mathcal{H}_h$ denote the Heegaard decomposition $S^3 = H_0
\cup_{\Sigma} H_1$ of genus $h \geq 1$ such that there is a system of
$2h$ curves $\alpha_i, \beta_i$ on $\Sigma'$ as in
Section~\ref{sec:ext} with the property that the $\beta_i$'s are
nullhomotopic in $H_0$ and the $\alpha_i$'s are nullhomotopic in
$H_1.$

With respect to the identification \eqref{nsigma}, the Lagrangians corresponding to $H_0$ and $H_1$ are given by 
$$L_0 = \{(A_1, B_1, \dots, A_h, B_h) \in G^{2h} \ | \ B_i = I, \ i=1,
\dots, h \},$$
$$L_1 = \{(A_1, B_1, \dots, A_h, B_h) \in G^{2h} \ | \ A_i = I, \ i=1,
\dots, h \}.$$
    
These have exactly one intersection point, the reducible $A_i = B_i =
I.$ Clearly $L_0$ and $L_1$ intersect transversely in $\nn(\Sigma')
\subset G^{2h}$ at that point. It is somewhat counterintuitive that
$L_0$ and $L_1$ can intersect transversely at $I$, because they both
live in the subspace $\Phi^{-1}(0)$ of codimension three in
$\nn(\Sigma').$ However, that subspace is not smooth, so there is no
contradiction. We conclude that the Floer chain group has one
generator; hence so does the homology.
\end {proof}

\begin {proposition}
For $h \geq 1$, we have an isomorphism
$$ \hsi(\#^h(S^1 \times S^2)) \cong  \bigl( H_*(S^3;
\zz/2\zz)\bigr)^{\otimes h},$$
where the grading of the latter vector
space is collapsed mod $8.$
\end {proposition}

\begin {proof}
Let $\mathcal{H}'_h$ be the Heegaard splitting of genus $h \geq 1$ for
$\#^h(S^1 \times S^2).$ Since $L_0 = L_1 \cong G^h \cong (S^3)^h,$ the cohomology ring of
$L_0$ is generated by its degree $d=3$ part. Under the monotonicity
assumptions which are satisfied in our setting, Oh \cite{OhSpectral}
constructed a spectral sequence whose $E^1$ term is $H_*(L_0;
\zz/2\zz)$ and which converges to $HF_*(L_0, L_0; \zz/2\zz).$ This
sequence is multiplicative by the results of Buhovski \cite{Buhovski}
and Biran-Cornea \cite{BiranCornea1}, \cite{BiranCornea2}. A
consequence of multiplicativity is that the spectral sequence
collapses at the $E_1$ stage provided that $N_L > d+1,$ 
see for example \cite[Theorem 1.2.2]{BiranCornea2}. This is satisfied
in our case because $N_{L_0}= N \geq 8.$ Hence $HF_*(L_0, L_0;
\zz/2\zz) \cong H_*(G^h; \zz/2\zz).$

Note that the results of Oh, Buhovski and Biran-Cornea were originally
formulated for monotone symplectic manifolds, i.e. in the setting of
Section~\ref{sec:mn}. However, they also apply to the Floer homology
groups defined in Section~\ref{sec:semi}. Indeed, the arguments in the
proof of Proposition~\ref{prop:semifloer} about the finiteness of the
Floer differential is finite and the fact that $\del^2 = 0$ apply
equally well to the ``string of pearls" complex used in
\cite{BiranCornea1}, \cite{BiranCornea2}.
\end {proof}

\begin {proposition}
For a lens space $L(p,q),$ with $\operatorname{g.c.d.}(p,q)=1,$ the symplectic instanton homology 
$\hsi(L(p,q))$ is a free abelian group of rank $p.$
\end {proposition}

\begin {proof}

Denote by $\mathcal{H}(p,q)$ the genus one 
Heegaard splitting of
$L(p,q)$.  In terms of the coordinates $A=A_1$ and $B=B_1$, the two
Lagrangians are given by $L_0 = {B = 1}$ and $L_1 = {A^pB^{-q}=1}.$
Their intersection consists of the space of representations
$\pi_1(L(p,q)) \cong \zz/p \to SU(2),$ which has several components:
when $p$ is odd, there is the reducible point $(A=B=I)$ and $(p-1)/2$
copies of $S^2;$ when $p$ is even, there are two reducibles ($A=B=I$
and $A=-I, B=I$) and $(p-2)/2$ copies of $S^2.$ It is straightforward
to check that each component is a clean intersection in the sense of
Po\'zniak \cite{Pozniak}. Therefore, there exists a spectral sequence
that starts at $H_*(L_0 \cap L_1) \cong \zz^p$ and converges to
$HF(L_0, L_1),$ cf. \cite{Pozniak}. Since the Euler characteristic of
$HF(L_0, L_1)$ is $p$ by Equation~\eqref{eq:euler}, the sequence must
collapse at the first stage.
\end {proof}

\begin {remark}
More generally, whenever we have a Heegaard decomposition
$\mathcal{H}$ of a three-manifold $Y$ with $H^1(Y) =0,$ the two
Lagrangians $L_0$ and $L_1$ will intersect transversely at the
reducible $I,$ cf. \cite[Proposition 1.1(c)]{AkbulutM}. We could then
fix an absolute $\zz/8\zz$-grading on $\hsi(\mathcal{H})$ by requiring
that the $\zz$ summand corresponding to $I$ lies in grading zero.
\end {remark}

\subsection {Comparison with other approaches}
\label {sec:comparisons}
Let $Y = H_0 \cup _{\Sigma} H_1$ be a Heegaard splitting of a
$3$-manifold, with $\Sigma$ of genus $h.$ Recall that the Lagrangians
$L_0 = L(H_0)$ and $L_1 = L(H_1)$ live inside the subspace
$$\Phi^{-1}(0) = \Bigl\{(A_1, B_1, \dots, A_h, B_h) \in G^{2h} \ \Big|
\prod_{i=1}^h [A_i, B_i] = I \Bigr\} \subset \nn(\Sigma').$$ There is
an alternative way of embedding $\Phi^{-1}(0)$ inside a symplectic
manifold of dimension $6h.$ Namely, let $\Sigma_+$ be the closed
surface (of genus $h+1$) obtained by gluing a copy of $T^2 \setminus
D^2$ onto the boundary of $\Sigma' = \Sigma \setminus D^2.$ Consider
the moduli space $\mm_\tw(\Sigma_+)$ of projectively flat connections
(with fixed central curvature) in an odd-degree $U(2)$-bundle over
$\Sigma_+$, as in Section~\ref{sec:degeneracies}:
$$\mm_{\tw}(\Sigma_+) =\Bigl\{(A_1, B_1, \dots, A_{h+1}, B_{h+1}) \in G^{2h+2} \ \Big| \prod_{i=1}^{h+1} [A_i, B_i] = -I \Bigr\}/ G.$$

Pick two particular matrices $X, Y \in G$ with the property that $[X,
  Y] = -I.$ Then we can embed $\Phi^{-1}(0)$ into
$\mm_{\tw}(\Sigma_+)$ by the map
$$ (A_1, B_1, \dots, A_h, B_h) \to [(A_1, B_1, \dots, A_h, B_h, X, Y)].$$

With respect to the natural symplectic form on $\mm_{\tw}(\Sigma_+),$
the spaces $L_0, L_1 \subset \Phi^{-1}(0)$ are still Lagrangians. One
can take their Floer homology, and obtain a $\zz/4\zz$ graded abelian
group. This was studied in \cite[Section 4.1]{WWField}, where it is
shown that it is a $3$-manifold invariant. It is not obvious how this
invariant relates to $\hsi.$

The advantage of using $\mm_{\tw}(\Sigma_+)$ instead of $\nn(\Sigma')$
is that the former is already compact (and monotone); therefore, the
definition of Floer homology is less technical and this allows one to
prove invariance. Nevertheless, the construction presented in this
paper (using $\nn (\Sigma')$) has certain advantages as well: first,
the resulting groups are $\zz/8\zz$-graded rather than
$\zz/4\zz$-graded. Second, it is better suited for defining an
equivariant version of symplectic instanton homology. Indeed, unlike
$\mm_{\tw}(\Sigma_+)$, the space $\nn(\Sigma')$ comes with a natural
action of $G$ that preserves the symplectic form and the
Lagrangians. Following the ideas of Viterbo from \cite{Viterbo},
\cite{Viterbo2}, we expect that one should be able to use this action
to define equivariant Floer groups $HSI^G_*(Y)$ in the form of
$H^*(BG)$-modules. For integral homology spheres, a suitable
Atiyah-Floer Conjecture would relate these to the equivariant
instanton homology of Austin and Braam \cite{AusBra2}.

In a different direction, it would be interesting to study the
connection between our construction and the Heegaard Floer homology
groups $\widehat{HF}, HF^+$ of Ozsv\'ath and Szab\'o \cite{HolDisk},
\cite{HolDiskTwo}. In particular, we ask the following:

\begin {question}
\label {quest}
For an arbitrary $3$-manifold $Y$, are the total ranks of $\hsi(Y) \otimes \qq
$ and $\widehat{HF}(Y) \otimes \qq$ equal?
\end {question}

Finally, we remark that Jacobsson and Rubinsztein \cite{JaRu} have
recently described a construction similar to the one in this paper,
but for the case of knots in $S^3$ rather than $3$-manifolds. Given a
representation of a knot as a braid closure, they define two
Lagrangians inside a certain symplectic manifold; this manifold was
first constructed in \cite{GHJW} and is a version of the extended
moduli space. Conjecturally, one should be able to take the Floer
homology of the two Lagrangians and obtain a knot invariant.

\bibliographystyle{abbrv}
\bibliography{biblio}

\end{document}